\documentclass[10pt,twoside,reqno]{amsart}
\usepackage{latexsym, amsmath, amssymb, longtable, booktabs,amscd}
\usepackage{amsmath,amssymb,amsthm}
\usepackage[overload]{empheq}
\usepackage{xcolor}
\usepackage{enumitem}
\usepackage[all]{xy}
\UseComputerModernTips
\numberwithin{equation}{section}
\usepackage[hidelinks]{hyperref} 
\newtheorem{thm}{Theorem}[section]
\newtheorem{theorem}[thm]{Theorem}
\newtheorem{cor}[thm]{Corollary}

\newtheorem{prop}[thm]{Proposition}
\newtheorem{lemma}[thm]{Lemma}
\theoremstyle{definition}

\theoremstyle{remark}
\newtheorem{rk}[thm]{Remark}

\newcommand{\N}{\mathbb{N}}
\newcommand{\kerp}{\mathrm{Ker} (P)}
\newcommand{\ind}{\mathrm{ind}}
\newcommand{\indkg}{\mathrm{ind}^{G}_{KZ}}

\newcommand{\gp}{\mathrm{Gal}\left(\bar{\mathbb{Q}}_p/\mathbb{Q}_p\right)}
\newcommand{\gl}{\mathrm{GL}_{2} (\mathbb{Q}_{p})}
\newcommand{\glfp}{\mathrm{GL}_{2} (\mathbb{F}_{p})}
\newcommand{\flap}{\lfloor \nu(a_p) \rfloor}

\newcommand{\symzp}{\mathrm{Sym}^r(\bar{\mathbb{Z}}_{p}^2)}
\newcommand{\symqp}{\mathrm{Sym}^r(\bar{\mathbb{Q}}_{p}^2)}

\begin{document}

\title{On the local constancy of certain mod $p$ Galois representations}
\author{Abhik {Ganguli}*}

\author{Suneel {Kumar}**}
\thanks {Email: *aganguli@iisermohali.ac.in,\ **suneelm145@gmail.com}
\thanks {\noindent Department of Mathematical Sciences, Indian Institute of Science Education and Research (IISER) Mohali, Sector 81, SAS Nagar,  Punjab-140306, India.}

\maketitle
\begin{abstract}
In this article we study local constancy of the mod $p$ reduction of certain $2$-dimensional crystalline representations of $\text{Gal}\left(\bar{\mathbb{Q}}_p/\mathbb{Q}_p\right)$ using the mod $p$ local Langlands correspondence.  We prove local constancy in the weight space by giving an explicit lower bound on the local constancy radius centered around weights going up to $(p-1)^{2} +3$ and the slope fixed in $(0,  \ p-1)$ satisfying certain constraints. We establish the lower bound by determining explicitly the mod $p$ reductions at nearby weights and applying a local constancy result of Berger.
\end{abstract}

\section{Introduction }
Let $p$ be an odd prime and $f$ be a normalized eigenform of weight $k\geq 1$, character $\psi$ and level $\Gamma_1(N)$ such that $p\nmid N$. The work of Deligne,  Deligne-Serre, Eichler-Shimura associates to $f$ a $p$-adic Galois representation $\rho_f: \text{Gal}\left(\bar{\mathbb{Q}}/\mathbb{Q}\right)\rightarrow \text{GL}_2(\bar{\mathbb{Q}}_p)$ such that $\rho_f$ is unramified at all primes $l\nmid pN$.  Further,  the characteristic polynomial of $\rho_f(\text{Frob}_l)$ (the arithmetic Frobenius at $l$) is given by $X^2-a_lX+l^{k-1}\psi(l)$, where $a_l$ is the Hecke operator $T_l$-eigenvalue of $f$. In the ordinary case,  a result of Deligne determines the mod $p$ reduction $\bar{\rho}_f|_{D_p}^{ss}$ (upto semisimplification) at the decomposition group $D_p \cong\gp$ over $p$ for weights $k\geq 2$. In the non ordinary case,  Fontaine and Edixhoven determine $\bar{\rho}_f|_{I_p}$ ($I_p$ is inertia at $p$) for $2\leq k\leq p+1$. Faltings proved that if $p\nmid N$ ($k\geq 2$) then $\rho_f|_{D_p}$ is a crystalline representation of Hodge-Tate weights $(0, k-1)$. Thus, purely local methods computing $\bar{\rho}_f|_{D_p}^{ss}$ exist that exploit the observation above. The problem of determining the mod $p$ reduction of $2$-dimensional crystalline representations of $\gp$ is a hard problem wherein the local techniques involve $p$-adic Hodge theory and more recently the mod $p$ local Langlands correspondence for $\text{GL}_{2} (\mathbb{Q}_{p})$ due to Breuil and Berger (\cite{Br03a}, \cite{Br03b}, \cite{BB10}, \cite{B10}).  Substantial work has been done using above local methods on computing the mod $p$ reduction in various ranges of slopes and weights (see for instance \cite{Br03b},\cite{BLZ},\cite{KBG},\cite{GG15},\cite{BG},\cite{BGR18},\cite{GhR}).  \\

In this article we consider the problem of local constancy of the $\bmod\ p$ reduction of certain $2$-dimensional crystalline representations of $\gp$.  Broadly speaking, we obtain local constancy in the weight space for weights $k$ up to $ (p-1)^{2} +3$ and the slope $\nu(a_p)$ fixed in $(0,  \ p-1)$ satisfying certain interdependency conditions (see Theorem \ref{intro} below). We determine an explicit radius of local constancy for these weights thereby obtaining a lower bound for the same. The key step in finding the radius above is the computation of the mod $p$ reduction of the crystalline representations that come from this neighbourhood of the weight using the mod $p$ local Langlands correspondence. \\

Let $p\geq 7$ be a prime and $\nu : \bar{\mathbb{Q}}^*_p\rightarrow \mathbb{Q}$ be the normalized valuation such that $\nu(p) = 1$.  Let $0\not=a_p\in\bar{\mathbb{Q}}_p$ with $\nu(a_p)>0$, and $k\geq 2$ be an integer.  Let $V_{k,a_p}$ be the irreducible, $2$-dimensional crystalline Galois representation of $\gp$ with Hodge-Tate weights $(0, \ k-1)$ such that $D_{cris}(V^*_{k,a_p})\cong D_{k,a_p}$ where $D_{cris}$ is the Fontaine's functor and $D_{k,a_p}$ is the admissible filtered module given in \cite{BLZ}. We note in passing that the crystalline Frobenius on $D_{k,a_p}$ has the characteristic polynomial $X^2-a_pX+p^{k-1}$.  Let $\bar{V}_{k,a_p}$ be the reduction of a $\gp$-stable lattice of $V_{k,a_p}$ upto semisimplification.  Our aim is to obtain local constancy of $\bar{V}_{k,a_p}$ in the weight space with a fixed positive slope $\nu(a_p)$. The evidence for local constancy is seen in results computing $\bar{V}_{k,a_p}$ for small slope.  From these results and Berger's theorem (Theorem B, \cite{Berger12}, \cite{Berger} or Theorem \ref{Berger} below) we expect local constancy to hold if $k$ and $k'$ are $p$-adically close enough and are in the same class modulo $p-1$.  \\

The first result giving an explicit upper bound for Berger's constant $m(k,a_p)$ is given in \cite{maam} for small weights with conditions on the slope similar to Theorem \ref{intro}.  
More precisely,  we write the weight $k$ in the form $b+c(p-1)+2$, where $b , c$ are assumed to be in the range $2\leq b\leq p-1, \ 0\leq c\leq 3$ respectively,  and such that $b\geq 2c$ and $k\not\equiv 3\bmod (p+1)$.  If the slope is in $(c, \frac{p}{2}+c)$ and weight $k>2\nu(a_p)+2$ it is shown in \cite{maam} that the Berger constant  $m(k,a_p)$ exists and bounded above by $2\nu(a_p)+1$. Our main result of this article is as follows:
\begin{thm}{\label{intro}}
Let $k =b+c(p-1)+2$ with $2\leq b\leq p$ and $0\leq c\leq p-2$. Fix $a_p$ such that $k>2\nu(a_p)+2$ and $c<\nu(a_p)<\text{min}\{\frac{p}{2}+c-\epsilon,\ p-1\}$ where $\epsilon$ is defined as in (\ref{dfn epsilon}). Further if $b \not\in\{2c+1, \ 2c-1, \ 2c-p, \ 2(c-1)-p\}$ and $(b,c)\not =(p,0)$ then the Berger's constant $m(k,a_p)$ exists such that $m(k,a_p)\leq\lceil 2\nu(a_p)\rceil+\epsilon+1$.  Moreover, $\bar{V}_{k',a_p}\cong \ind\left(\omega^{k-1}_2\right)$ for all $k' \in k+p^{t}(p-1)\mathbb{Z}^{\geq 0}$,  where $ t \geq \lceil 2\nu(a_p)\rceil+\epsilon$.
\end{thm}

We take $p\geq 7$ in order to apply Berger's theorem in Corollary \ref{vkap}. In our theorem, the conditions $2c\leq b\leq p-1$ and $c\leq 3$ in \cite{maam} are no longer there to include all $2\leq b\leq p$ and $0\leq c\leq p-2$ with above constraints.  For example, with $c = p-2$ we cover all the values of $b$ for sufficiently large slope $p-2<\nu(a_p)<p-1$.  More precisely,  if $c\geq \frac{p}{2}+1$ then one could take the slope $\nu(a_p)$ in $(c,\ p-1)$.  Note that the upper bound of $p/2 +c$ for the slope in \cite{maam} is assumed to be at most $p-1$.  This is because with $\nu(a_p)< p-1$ (and $k-2 > 2 \nu(a_p)$) one is able to apply Lemma $3.2$ in \cite{maam} (Lemma \ref{lm3.2} below).  In the theorem above the lower bound on $k$ holds in any case for $c\geq 2$.  Allowing for $c\geq 3$ and $b\leq 2c-2$ in our result makes the analysis significantly more involved, also revealing interesting phenomena discussed below.\\
 
The approach in \cite{maam} and our result is to show that the surjection $P: \indkg \left(V_r\right)\rightarrow\bar{\Theta}_{k',a_p}$ factors through a successive quotient $\indkg \left(\frac{V^{(n)}_r}{V^{(n+1)}_r}\right)$ for $k' = r+2 \in k+p^{t}(p-1)\mathbb{Z}^{> 0}$, and for some $n \leq \lfloor \nu(a_p) \rfloor$ (see (\ref{filt})).  Using the mod $p$ local Langlands correspondence,  we obtain the result above in the generic irreducible case (Proposition \ref{final prop}).  In \cite{maam}, $n$ remains constant and is equal to $c$ where the hypothesis $b\geq 2c$ plays a crucial role.  Interestingly in our case, for a fixed $c$, $n$ varies accordingly as $b$ lies in $[2, \ 2c-2-p-1], [2c-2-p,\ 2c-2]$ or $[2c-1, \ p]$. More precisely, $n = c-\epsilon$ (if $(b, c) \neq (p, 0)$, Theorem \ref{combining}) where $\epsilon$ is as defined in \eqref{dfn epsilon}.  We show that all the Jordan Holder factors coming from $ \indkg \left(\frac{V^{(m)}_r}{V^{(m+1)}_r}\right)$ where $0\leq m \leq \flap$ and $m\not=n$ do not contribute to $\bar{\Theta}_{k, a_p}$. 
In fact, our proof splits naturally into two parts: $0\leq m<n$ and $n<m\leq \flap$ with a substantial difference in the analysis treating these two regimes.  A crucial observation in \cite{maam} (Lemma \ref{lm3.2} below) is that the successive quotients $ \frac{V^{(m)}_r}{V^{(m+1)}_r}$ are generated by the polynomial $F_m(x,y)$.  \\

In Propositions \ref{mono 1} \& \ref{mono 1.2}, we give a family of monomials say $\{P_{a , m}\}$ for each $1 \leq m < c- \epsilon$ that are in $\kerp$.  A subset of these monomials are used in Proposition \ref{other generator} to obtain (for each $m$) a family of monomials $\{Q_{a, m}\}$ depending also on $\epsilon$ that are $F_m(x,y)$ (up to a unit) modulo $V_{r}^{(m+1)}+\kerp$.  We next obtain a monomial that is in both $\{P_{a , m}\}$ and $\{Q_{a, m}\}$ for each $1 \leq m < c- \epsilon$.   
Exploiting this we show (with some technical computation for $m = 0$) that the successive quotients do not contribute to $\bar{\Theta}_{k',a_p}$ for $0\leq m<c-\epsilon$  (Proposition \ref{m<slope}, Lemma \ref{lm m<c}).  The cases $m > c - \epsilon$ differ from above in that the polynomial $F_m(x,y)$ is directly shown to be in $\kerp$ (Proposition \ref{m>c} \& Lemma \ref{lm m>c}).  We further use the monomials $P_{a , m}$ to eliminate the possibility of $\bar{V}_{k',a_p}$ being reducible in some non generic cases.   \\


To generate the monomials $P_{a , m}$ (and $F_m$ if $m > c - \epsilon$) we consider for each $m$ a matrix $A$ over $\mathbb{Z}_{p}$ coming naturally from Propositions \ref{gen 1} \& \ref{gen 2}, and show that the columns space of $A$ contains the vectors needed to generate these monomials. The matrices are typically of size $c \times c$ ($1 \leq c \leq p-2$) with entries given by products of binomial coefficients, and thus require proving certain binomial identities (see \S3). 
The computations involving the Hecke operator $T$ in Propositions \ref{gen 1} \& \ref{gen 2} require a delicate choice of functions to obtain the specific polynomials. The lower bounds $\lceil 2\nu(a_p)\rceil+\epsilon$ and $2\nu(a_p)+2$ on $t$ and $k$ respectively play a crucial role in making certain terms vanish modulo $p$ in the above computations.  The vanishing of these terms also require the mod $p$ congruences in Lemma \ref{srjm} as well as the precise valuations of binomial coefficients in Lemma \ref{rk3.15}. \\

For the weight $k$ in our range we have $\lfloor\frac{k-2}{p-1}\rfloor = c<\nu(a_p)$ barring a few exceptions.  Therefore,  \cite{BLZ} implies $\bar{V}_{k,a_p}\cong\ind(\omega^{k-1}_2)$ whenever $(p+1)\nmid (k-1)$ and reducible otherwise.  Using this fact together with the mod $p$ local Langlands correspondence, one can predict the integer $n$ in Proposition \ref{vrc 1}. Theorem \ref{combining} and Proposition \ref{vrc 1} imply that the reducible cases can occur only if $b\in\{2c+1, 2c-1, 2c-3, 2c-p, 2c-2-p, 2c-4-p\}$ or if $(b,  c)\in\{(p-2, 0), (p, 0),(p,1)\}$.  If there is local constancy,  we expect from \cite{BLZ} that $\bar{V}_{k,a_p}$ always be reducible if $b \in\{ 2c-1, \ 2(c-1)-p\}$ or $(b,c) =(p,0)$ (indeed $(p+1) | (k-1)$ only in these cases),  and be irreducible in all other cases.  In Proposition \ref{final prop} we show that if $b\in\{ 2c-3, 2c-4-p\}$ or $(b, c)\in\{(p-2,  0), (p,1)\}$ then $\bar{V}_{k,a_p}$ is indeed irreducible.  We intend to report soon on the remaining exceptional cases in our ongoing work.\\

We now discuss some results on local constancy extant in the literature.  The overlap of the constraints in our theorem with that of the zig-zag conjecture for slope $\nu (a_{p}) = 3/2$ proven in Theorem $1.1$ in \cite{gv} is precisely when $b=3$, $c=1$, and so $k = p+4$.  Indeed in this situation,  we deduce local constancy for $t\geq 1$ from \cite{gv} with the reduction given as $\ind \ (\omega_{2}^{k-1})$.  We note that this weight lies in the exceptional case $b = 2c+1$, the reduction being irreducible and compatible with \cite{BLZ}.  Further from \cite{gv},  we observe that if $c=0$ (i.e., $k =5$) and $p \geq 7$,  local constancy is violated when $a_p = p^{3/2}$ whereas local constancy is preserved (for $t \geq 2$) when one chooses $a_p = p^{3/2} u$ with $ u^2 = 1 + p^{1/2}$ (see also \S $1.3$, \cite{cgy} for the case $k=4$ and $a_p =p$).  This illustrates the subtle phenomenon of the existence of local constancy depending on $a_p$ and $k$.  We also refer to Theorem $9.2.1$ in \cite{sandra} which gives an algorithm to compute effectively a radius of local constancy for small weight $k$ and prime $p$. The result in Corollary $1.12$ of \cite{GhR} can be seen proving local constancy in a regime that has very little overlap with our result which requires the BLZ condition $c<\nu(a_p)$.  Indeed the only common cases are when $c=0$ (with $r_0 = b$) or $k = 2p+1$ (i.e., $c=1, \ b=p$) wherein both results give the same reduction.

\section{Background}
\subsection{The mod $p$ local Langlands correspondence} 
We begin by recalling some notations and definitions.  We fix an algebraic closure $\bar{\mathbb{Q}}_{p}$ of $\mathbb{Q}_{p}$ with the ring of integers $\bar{\mathbb{Z}}_{p}$ and the residue field $\bar{\mathbb{F}}_{p}$. Let $G_p$ and $G_{p^2}$ be the absolute Galois groups of $\mathbb{Q}_p$ and $\mathbb{Q}_{p^2}$ respectively where $\mathbb{Q}_{p^2}$ is the unique unramified quadratic extension of $\mathbb{Q}_p$.  Let $\omega_1 =\omega$ be the mod $p$ cyclotomic character, and $\omega_2$ be a fixed fundamental character of level $2$. We view $\omega_1$ and $\omega_2$ as characters of $\mathbb{Q}^*_p$ via local class field theory (identifying uniformizers with geometric Frobenii). For $a\in\mathbb{Z}^{\geq 0}$ such that $(p+1)\nmid a$ let $\ind(\omega^a_2)$ denote the unique two dimensional irreducible representation of $G_p$ with determinant $\omega^a$ and whose restriction to inertia is isomorphic to $\omega^a_2\oplus\omega^{ap}_2$.\\

We denote the group $\gl$ by $G$, its maximal compact subgroup $\mathrm{GL}_2\left(\mathbb{Z}_p\right)$ by $K$ and the center of $G$ by $Z\cong \mathbb{Q}^*_p$.  For $r\geq 0$  let $V_r: = \text{Sym}^r(\bar{\mathbb{F}}_{p}^2)$ be the symmetric power representation of $\mathrm{GL}_2(\mathbb{F}_p)$ of dimension $r+1$.  We can also view $V_r$ as representations of $KZ$ by defining the action of $K$ through the natural surjection $K\twoheadrightarrow \mathrm{GL}_2(\mathbb{F}_p)$, and by letting $p$ act trivially. 
For $0\leq r\leq p-1, \ \lambda\in\bar{\mathbb{F}}_p$ and a smooth character $\eta :\mathbb{Q}^*_p\rightarrow \bar{\mathbb{F}}^*_p$,  the representation
$$\pi(r, \lambda, \eta) := \frac{\indkg(V_r)}{T-\lambda}\otimes(\eta\circ\det)$$
is a smooth admissible representation of $G$ where $\indkg$ denotes compact induction (see \cite{Br03a},  \cite{KBG}).  The operator $T$ (see \S \ref{Hecke}) is the Hecke operator $T_p$ generating the Hecke algebra\linebreak $ \text{End}_{G} ( \indkg (V_{r})) = \bar{\mathbb{F}}_{p} [T_p]$. These representations give all the irreducible smooth admissible representations of $G$ (\cite{BL94},\cite{BL95},\cite{Br03a}).  For $\lambda\in\bar{\mathbb{F}}_p$, let $\mu_\lambda$ be the unramified character of $G_p$ that sends the geometric Frobenius to $\lambda$. Then Breuil 's semisimple $\bmod\ p$ local Langlands correspondence $LL$ (see \cite{Br03b}) is as follows:
\begin{itemize}
\item $\lambda = 0:$ \hspace{5em} $\ind(\omega^{r+1}_2)\otimes\eta\xleftrightarrow{ LL}  \pi(r, 0 , \eta)$\\
\item $\lambda\not=0:$\hspace{5em} $(\mu_\lambda\omega^{r+1}\oplus\mu_{\lambda^{-1}})\otimes\eta\xleftrightarrow{LL} \pi(r,\lambda,\eta)^{ss}\oplus\pi([p-3-r], \lambda^{-1},\omega^{r+1}\eta)^{ss}$ \\
where $\{0, 1, ..., p-2\}\ni [p-3-r]\equiv p-3-r\bmod (p-1)$.\\
\end{itemize}

For integers $k\geq 2$ we define $\Pi_{k,a_p}: = \frac{\indkg(\symqp)}{T-a_p}$ as representations of $G$ where $r=k-2$ and $T$ is the Hecke operator from \S \ref{Hecke}. We consider the $G$-stable lattice $\Theta_{k,a_p}$ in the irreducible representation $\Pi_{k,a_p}$ (see  \cite{Br03b}, \cite{BB10}) given by 
$$\Theta_{k,a_p}:= \text{image}\left(\indkg(\symzp)\rightarrow\Pi_{k,a_p}\right)\cong\frac{\indkg(\symzp)}{(T-a_p)\indkg(\symqp)\cap\indkg(\symzp)}.$$
By the compatibility of the $p$-adic and mod $p$ local Langlands correspondences (\cite{Br03b},\cite{B10}, \cite{BB10}) we have 
$$\bar{\Theta}^{ss}_{k,a_p}\cong LL(\bar{V}_{k,a_p})\quad \text{where}\quad\bar{\Theta}_{k,a_p}:=\Theta_{k,a_p}\otimes \bar{\mathbb{F}}_p.$$
Since the $\bmod\ p$ local Langlands correspondence is injective,  to determine $\bar{V}_{k,a_p}$ it is enough to compute $\bar{\Theta}_{k,a_p}^{ss}.$\\ 

\subsection{Hecke Operator T}\label{Hecke} We give an explicit definition of the Hecke operator $T = T_p$ below (see \cite{Br03b} for more details). For $m = 0$,  set $I_0 = \{0\}$  and for $m>0$, let $I_m = \{[\lambda_0]+p[\lambda_1]+...+p^m[\lambda_{m-1}]\ |\ \lambda_i\in\mathbb{F}_p\}\subset\mathbb{Z}_p$ where square brackets denote Teichm\"{u}ller representatives.  For $m\geq 1$ there is a truncation map $[\ ]_{m-1}: I_m\rightarrow I_{m-1}$ given by taking the first $m-1$ terms in the $p$-adic expansion above.  For $m = 1$, $[\ ]_{m-1}$ is the zero map.  For $m\geq 0$ and $\lambda\in I_m$,  let
$$g^0_{m, \lambda} = \begin{pmatrix}
p^m & \lambda\\
0 & 1
\end{pmatrix}
\quad\text{and }\quad
g^1_{m, \lambda} = \begin{pmatrix}
1 & 0\\
p\lambda & p^{m+1}
\end{pmatrix}. $$
Then we have
$$ G =\underset{\substack{m\geq 0, \lambda\in I_m\\ i\in\{0,1\}}} \coprod KZ(g^i_{m,\lambda})^{-1}.$$
Let $R$ be a $\mathbb{Z}_p$-algebra and $V = \text{Sym}^rR^2$ be the symmetric power representation of $KZ$, modelled on homogeneous polynomials of degree $r$ in the variables $x$ and $y$ over $R$. For $g\in G, \ v\in V$,  let $[g, \ v]$ be the function defined by: $[g, v](g') = g'g\cdot v$ for all $g'\in KZg^{-1}$ and zero otherwise.  Since an element of $\indkg (V)$ is a $V$-valued function on $G$ that has compact support  modulo $KZ$, one can see that every element of  $\indkg(V)$ can be written as a finite sum of $[g, v]$ with $g= g^0_{m\lambda}$ or $g = g^1_{m, \lambda}$, for some $\lambda\in I_m$ and $v\in V$. Then the action of $T$ on $[g, v]$ can be given explicitly when $g = g^0_{n,\mu}$ with $n\geq 0$ and $\mu\in I$.  Let $v = \sum\limits_{j = 0}^{r}c_jx^{r-j}y^j$,  with $c_j\in R$.  We write $T = T^++T^-$ where
\begin{eqnarray*}
T^+([g^0_{n,\mu},v]) &=& \underset{\lambda\in I_1}\sum \left[g^0_{n+1,\mu+p^n\lambda}, \sum\limits_{j = 0}^{r}p^j\left(\sum\limits_{i= j}^{r}c_i{i\choose j}(-\lambda)^{i-j}\right)x^{r-j}y^j\right]\\
T^-([g^0_{n,\mu}, v]) &=& \left[g^0_{n-1, [\mu]_{n-1}}, \sum\limits_{j = 0}^{r}\left(\sum\limits_{i= j}^{r}p^{r-i}c_i{i\choose j}\left(\frac{\mu-[\mu]_{n-1}}{p^{n-1}}\right)^{i-j}\right)x^{r-j}y^j\right]\quad \text{for}\ n>0\\
T^-([g^0_{n,\mu}, v]) &=& \left[\alpha, \sum\limits_{j =0}^{r}p^{r-j}c_jx^{r-j}y^j\right]\quad\text{for}\ n=0,\ \text{where}\ \alpha: = g^1_{0,0}.
\end{eqnarray*}
\subsection{The filtration} 
Let $r = k'-2\geq (\nu +1)(p+1)$, where $\nu: = \lfloor\nu(a_p)\rfloor$.  From the definition of $V_r$ and $\bar{\Theta}_{k',a_p}$ it follows that there is a natural surjection
 $$P: \indkg(V_r)\twoheadrightarrow\bar{\Theta}_{k',a_p}.$$
Now let us consider the Dickson polynomial $ \theta := x^py-xy^p\in V_{p+1}.$ Here we note that $\text{GL}_2(\mathbb{F}_p)$ acts on $\theta$ by the determinant character. For $m\in\mathbb{N}$, let us denote 
$$V^{(m)}_r = \{f\in V_r\ | \ \theta^m \ \text{divides}\ f\ \text{ in}\ \bar{\mathbb{F}}_p[x,y]\}$$
which is a subrepresentation of $V_r$.  By using Remark $4.4$ of \cite{KBG}, one can see that the map $P$ factors through $\indkg\left(\frac{V_r}{V^{(\nu+1)}_r}\right)$, where $\nu: = \lfloor\nu(a_p)\rfloor$.  So let us consider the following chain of submodules
\begin{eqnarray}{\label{filt}}
0\subseteq\indkg\left(\frac{V^{(\nu)}_r}{V^{(\nu +1)}_r}\right)\subseteq \indkg\left(\frac{V^{(\nu-1)}_r}{V^{(\nu+1)}_r}\right)\subseteq ...\subseteq\indkg\left(\frac{V_r}{V^{(\nu+1)}_r}\right).
\end{eqnarray}
 For $0\leq m\leq \nu$,  observe that $\indkg\left(\frac{V^{(m)}_r}{V^{(m+1)}_r}\right)$ are the successive quotients in the above filtration.  In the following two lemmas we make precise the notion of a successive quotient not contributing to $\bar{\Theta}_{k',a_p}$ via the map $P$. 
 
\begin{lemma}{\label{lm m<c}}
Let $1\leq n\leq \nu : = \lfloor\nu(a_p)\rfloor$ and assume for $0\leq m\leq n-1$ that there exists $W_m\subset V^{(m)}_r$ such that $P\left(\ind^{G}_{KZ}(W_m)\right) = 0$ and $W_m\twoheadrightarrow \frac{V^{(m)}_r}{V^{(m+1)}_r}$. Then the map $P$ restricted to $\ind^{G}_{kZ}\left(\frac{V^{(n)}_r}{V^{(\nu+1)}_r}\right)$ is a surjection.
\end{lemma}

\begin{lemma}{\label{lm m>c}}
Let $1\leq n\leq \nu :=  \lfloor\nu(a_p)\rfloor$ and suppose for $n\leq m\leq \nu$ that there exists $G_m(x,y)\in V_r$ such that $P([g, \ G_m(x,y)]) = 0$. If $G_m(x,y)$ generates $\frac{V^{(m)}_r}{V^{(m+1)}_r}$ then the map $P$ factors through  $\ind^{G}_{kZ}\left(\frac{V_r}{V^{(n)}_r}\right)$.
\end{lemma}

Next, we determine the Jordan-Holder factors of the successive quotients $\frac{V^{(n)}_r}{V^{(n+1)}_r}$.  Let us write $r-n(p+1) = r'+d'(p-1)$ such that $p\leq r'\leq 2p-2$ and for some $d'\in\mathbb{Z}^{\geq 0}$.  By $(4.1)$ and $(4.2)$ of \cite{glover} together with Lemma $5.1.3$ of \cite{Br03b} gives:\\
(i) if $r' = p$ then
\begin{eqnarray}{\label{r' = p}}
0\longrightarrow V_1\otimes D^n\longrightarrow\frac{V^{(n)}_r}{V^{(n+1)}_r}\longrightarrow V_{p-2}\otimes D^{n+1}\longrightarrow 0.
\end{eqnarray}
The first map sends $(x, \ y)$ to $(\theta^n x^{r-n(p+1)}, \ \theta^n y^{r-n(p+1)})$ and the second map sends $\theta^n x^{r-n(p+1)-1}y$ to $x^{p-2}$.\\
(ii) if $r'\not = p$ then 
\begin{eqnarray}{\label{r' not p}}
0\longrightarrow V_{r'-(p-1)}\otimes D^n\longrightarrow\frac{V^{(n)}_r}{V^{(n+1)}_r}\longrightarrow V_{2(p-1)-r'}\otimes D^{n+r'-(p-1)}\longrightarrow 0.
\end{eqnarray}
The first map sends $(x^{r'-(p-1)}, \ y^{r'-(p-1)})$ to $(\theta^n x^{r-n(p+1)}, \ \theta^ny^{r-n(p+1)})$ because ${{r'}\choose p-1}\equiv 0\bmod p$ as $1\leq r'-p\leq p-2$. For $r'-(p-1)\leq i\leq p-1$, the second map sends $\theta^nx^{r-n(p+1)-i}y^i$ to $\alpha_i\ x^{p-1-i}y^{p-1-r'+i}$ where $\alpha_i:= (-1)^{r'-i}{{2(p-1)-r'}\choose p-1-r'+i}\not\equiv 0\bmod p$ because $0\leq 2(p-1)-r'\leq p-3$ and $0\leq p-1-r'+i\leq 2(p-1)-r'$.\\

\subsection{Theorem of Berger and a crucial lemma} 
\begin{theorem}[Berger \cite{Berger12}, \cite{Berger}]{\label{Berger}}
 Suppose $a_p\not= 0$ with $\nu (a_{p}) > 0$ and $k>3\nu(a_p)+\frac{(k-1)p}{(p-1)^2}+1$ then there exist $m = m(k,a_p)$ such that $\bar{V}_{k',a_p}\cong \bar{V}_{k,a_p}$ if $k'-k\in p^{m-1}(p-1)\mathbb{Z}_{\geq 0}$. 
\end{theorem}

For integers $0\leq m\leq s$ let us define polynomials $F_m$ in  $V_r$ as follows
$$F_m(x,y) := x^my^{r-m}-x^{r-s+m}y^{s-m}$$
where $r>s$ and $r\equiv s\bmod (p-1)$.
\begin{lemma}[Bhattacharya, Lemma 3.2, \cite{maam}]{\label{lm3.2}}
Let $r\equiv s\bmod (p-1)$, and $t = \nu(r-s)\geq 1$ and $1\leq m\leq p-1$.\begin{enumerate}
\item For $s\geq 2m$, the polynomial $F_m$ is divisible by $\theta^m$ but not by $\theta^{m+1}$.
\item For $s>2m$, the image of $F_m$ generates the subquotient $\frac{V^{(m)}_r}{V^{(m+1)}_r}$ as a $\glfp$-module.
\end{enumerate}
\end{lemma}

\subsection{Notations and Conventions} 
We fix the following conventions in the rest of this article unless stated otherwise:
\begin{enumerate}
\item The integer $p$ always denotes a prime number greater than equal to $7$.  The integers $b$ and $c$ are from $\{2, 3, ...,p\}$ and $\{0,1, ..., p-2\}$ respectively.
\item We define $\epsilon$ as follows 
\begin{eqnarray}{\label{dfn epsilon}}
 \epsilon =\begin{cases} 
 0 & \text{if}\quad 2c-1\leq b\leq p\\
 1 & \text{if} \quad 2(c-1)-p\leq b\leq 2(c-1)\\
 2 & \text{if}\quad 2\leq b\leq 2(c-1)-(p+1).
\end{cases}
\end{eqnarray}
\item We write $s = b+c(p-1)$ and $r = s+p^t(p-1)d$ with $p\nmid d$,  and $t,d\in\mathbb{N}$ and so $s < r$.
\item For $n\in\mathbb{Z}^{\geq 0}$ and $k\in\mathbb{Z}$,  we define ${n\choose k} = 0$ if $k > n$ or $k < 0$ and the usual binomial coefficient otherwise.
\item For $A\equiv B$, where $A,B\in \text{M}_n(\mathbb{Z}_p)$ we mean that $A\equiv B\bmod p$.
\item Unless stated otherwise,  for $A, B\in\indkg(\symqp)$, by $A\equiv B$ or $A\equiv B\bmod p$ we mean that $A- B$ is in $\textbf{m}_{\bar{\mathbb{Z}}_p}\indkg(\symzp)$.
\item By the vectors $\{\textbf{e}_j\}$ we mean the standard basis of a free module over $\mathbb{Z}_p$. 
\item For $v\in \text{Sym}^r(\bar{\mathbb{F}}^2_p)$, by $v\in\kerp$ we mean $[id,\ v]\in\kerp$.
\end{enumerate}

\section{Some Binomial Identities }
\begin{lemma}{\label{cmbi 4}}
Let $c,m,b,k\in\N\cup\{0\}$ and $m\leq b-c,\ k\geq 1$ then
$$\sum_{0\leq i\leq k}(-1)^i{{b-m-c+1}\choose i}{{b-m-c+k-i}\choose b-m-c} = 0$$ 
$$\text{and}\quad \sum_{0\leq l\leq c}(-1)^{c-l}{{b-m-c+1}\choose {b-m-c-l}}{{b-m-l}\choose {c-l}} = (-1)^c{{b-m+1}\choose b-m-c}. $$
\end{lemma}
\begin{proof}
Consider the following 
$$(x-1)^{b-m-c+1}x^{k-1} = \sum_{0\leq i\leq b-m-c+1}(-1)^i{{b-m-c+1}\choose i}x^{b-m-c+k-i}.$$
On differentiate with respect to $x,(b-m-c)$ time, putting $x = 1$ and dividing by $(b-m-c)!$,  gives
$$\sum_{0\leq i\leq b-m-c+1}(-1)^i{{b-m-c+1}\choose i}{{b-m-c+k-i}\choose b-m-c} = 0.$$
Observe $b-m-c+k-i\geq 0 \ \ \forall \ \ i$, and if $k< b-m-c+1$ then ${{b-m-c+k-i}\choose b-m-c} = 0 \ \ \forall  \ \ i\geq k+1$ and if $k>b-m-c+1$ then ${{b-m-c+1}\choose i} = 0 \ \ \forall\ \  i>b-m-c+1$. Therefore above summation runs over $0$ to $k$ so first part is done.\\
Now for the second part, we put $l = i-1$, and so we need to prove the following 
\begin{align*}
& &\sum_{1\leq i\leq c+1} (-1)^{c+1-i}{{b-m-c+1}\choose i}{{b-m+1-i}\choose b-m-c} &=& (-1)^c{{b-m+1}\choose b-m-c} \hspace{8em}\\
& \iff & \ \ \ \ \sum_{0\leq i\leq c+1} (-1)^{c+1-i}{{b-m-c+1}\choose i}{{b-m+1-i}\choose b-m-c} &=& 0 \hspace{14em}\\
& \iff & \ \ \ \ \sum_{0\leq i\leq c+1} (-1)^i{{b-m-c+1}\choose i}{{b-m+1-i}\choose b-m-c} &=& 0 \hspace{14em}
\end{align*}
which is part one of this Lemma for $k = c+1$.
\end{proof}

\begin{lemma}{\label{cmbi 1}} For every $j,m\in \N$ we have 
$$\sum_{1\leq i\leq j}(-1)^{i+1}{{m+1}\choose i}{{m+j-i}\choose j-i} = {{m+j}\choose j}.$$ 
\end{lemma}
\begin{proof}
We prove Lemma by induction on $j$. For $j = 1$ result follows trivially. By induction assume result is true for $1\leq j\leq k$ and need to prove $j= k+1 $.
Now \\
\begin{eqnarray*}
{{m+k+1}\choose k+1}&=&\frac{(m+k+1)}{k+1}{{m+k}\choose k}\\
&=&\frac{(m+k+1)}{k+1}\sum_{1\leq i\leq k}(-1)^{i+1}{{m+1}\choose i}{{m+k-i}\choose k-i}\\
&=& \sum_{1\leq i\leq k}(-1)^{i+1}{{m+1}\choose i}\left(\frac{(m+k+1-i)}{k+1}+\frac{i}{k+1}\right){{m+k-i}\choose k-i}\\
&=& \sum_{1\leq i\leq k}(-1)^{i+1}{{m+1}\choose i}\left(\frac{(k+1-i)}{k+1}{{m+k+1-i}\choose k+1-i}+\frac{i}{k+1}{{m+k-i}\choose k-i}\right)\\
&=& \sum_{1\leq i\leq k}(-1)^{i+1}{{m+1}\choose i}{{m+k+1-i}\choose k+1-i}-\sum_{1\leq i\leq k}(-1)^{i+1}\frac{i}{k+1}{{m+1}\choose i}{{m+k-i}\choose k+1-i}.
\end{eqnarray*}
So to prove our result we need to prove following
\begin{eqnarray*}
& &-(-1)^k{{m+1}\choose k+1}-\sum_{1\leq i\leq k}(-1)^{i+1}\frac{i}{k+1}{{m+1}\choose i}{{m+k-i}\choose k+1-i} = 0\\
&\iff & \sum_{1\leq i\leq k}(-1)^{i+1}{m\choose i-1}{{m+k-i}\choose k+1-i}+(-1)^k{m\choose k} = 0\\
&\iff & \sum_{0\leq i\leq k-1}(-1)^i{m\choose i}{{m+k-1-i}\choose k-i}+(-1)^k{m\choose k}= 0\quad\text{by replacing $i-1$ by $i$ }\\
&\iff & \sum_{0\leq i\leq k}(-1)^i{m\choose i}{{m+k-1-i}\choose m-1}= 0.
\end{eqnarray*}
Now we consider  the following 
\begin{eqnarray*}
(x-1)^m x^{k-1} = \sum_{0\leq i\leq m}(-1)^i{m\choose i}x^{m+k-1-i}
\end{eqnarray*}
differentiate with respect to $x,(m-1)$ time, divide by $(m-1)!$ and putt $x=1$
\begin{eqnarray*}
\sum_{0\leq i\leq m}(-1)^i{m\choose i}{{m+k-1-i}\choose m-1}= 0.
\end{eqnarray*}
If $k\leq m,\ m-1+k-i<m-1 \ \ \forall \ \ i\geq k+1\ \Rightarrow {{m+k-1-i}\choose m-1} = 0$. If $k>m$ then for $m+1\leq i\leq k\Rightarrow {m\choose i} = 0$. So in all the cases we got our result.
\end{proof}

Suppose $r\equiv s\bmod p^t(p-1)$ for some $s = b+c(p-1), \ t:=\nu(r-s)>0$.  And for $0\leq i\leq s-l,\linebreak \ \ 0\leq m\leq p-1,\ 0\leq l\leq p-1$ define 
 \begin{equation}{\label{dfnsrjm}}
 S_{r,i,l,m} :=\underset{\substack{s-m\leq j<r-m\\j\equiv (r-m)\bmod (p-1)}}\sum{{r-l}\choose j}{j\choose i}.
 \end{equation} 
 
\begin{lemma}{\label{srjm}}
 Let $r = s+dp^t(p-1)$ with $p \not|d$ for some $s = b+c(p-1),2\leq b\leq p$
for $0\leq c\leq p-1$.  Let $0\leq l\leq p-1$ and $0\leq m\leq p-1$ such that $s-l\geq 0$ and $s-m\geq 0$.  Then for $0\leq i\leq s-l$ we have 
$$S_{r,i,l,m}\equiv \begin{cases}
                                \underset{\substack{i\leq j<s-m}}\sum {{r-l}\choose i}\left({{s-l-i}\choose j-i}-{{r-l-i}\choose j-i}\right)\bmod p^t\quad \text{if}\quad i<s-m,\ 0\leq l\leq c\\
                                0 \bmod p^t\quad\text{if} \quad i = s-m,\ l\leq m\\
                                -{{r-l}\choose r-m}{{r-m}\choose i}\bmod p^t\quad\text{if}\quad i>s-m,\ l\leq m.
                                  \end{cases}$$
Further assume $0\leq i\leq \text{min}\{s-l,\ s-m\}$ (so that we are always in first two case) then we have 
\[
S_{r,i,l,m}\equiv \begin{cases}
                         0\bmod p^t \  & \ \text{if} \ \ c=0\\
                         0\bmod p^{t-(c-1)} \ & \ \text{if} \ \ c\geq 1 \ \& \ 2\leq b\leq p-1\\
                         0\bmod p^{t-(c-1)} \ & \ \text{if}\ \ c+m\geq 2,\ c\geq 1 \ \& \ \ b=p\\
                         0\bmod p^{t-c} \ \ & \ \text{if}\ \ c+m<2\ ,c\geq 1\ \&\ \ b = p.                         
                         \end{cases}
\]
\end{lemma}
\begin{proof}
Expend binomial expansion 
$$(1+x)^{r-l} = \sum_{0\leq j\leq r-l}{{r-l}\choose j}x^j$$
differentiating above with respect to $x, i^{th}$ time, dividing by $i!$ and multiply by $x^{i-(s-m)}$
\begin{align*}
{{r-l}\choose i}(1+x)^{r-l-i}x^{i-(s-m)} &=& \sum_{i\leq j\leq r-l}{{r-l}\choose j}{j\choose i}x^{j-(s-m)}\hspace{17em}\\
(1+x)^{r-l-i}x^{i-(s-m)} &=& \sum_{i\leq j\leq r-l}{{r-l-i}\choose j-i}x^{j-(s-m)}\hspace{17em}\\
\sum_{\zeta\in\mu_{p-1}}(1+\zeta)^{r-l-i}\zeta^{i-(s-m)} &=& \underset{\substack{i\leq j\leq r-l\\j\equiv (s-m)\bmod (p-1)}}\sum{{r-l-i}\choose j-i}(p-1).\hspace{15em}
\end{align*}
Similarly we have the following 
\begin{align*}
&\sum_{\zeta\in\mu_{p-1}}(1+\zeta)^{s-l-i}\zeta^{i-(s-m)} &=& \sum_{i\leq j\leq s-l,j\equiv (s-m)\bmod (p-1)}{{s-l-i}\choose j-i}(p-1).
\end{align*}
Note that for $\zeta\not= -1 $, $(1+\zeta)^{p-1} \equiv\ 1\bmod p \ \ \ \implies \ \ \ (1+\zeta)^{p-1}= 1+pz$ where $z\in\mathbb{Z}_p$. Therefore $(1+\zeta)^{(r-s)}\equiv\ 1\bmod p^{t+1}$. Hence we have \\
\begin{eqnarray*}
\sum_{\zeta\in\mu_{p-1}\setminus \{-1\}}(1+\zeta)^{s-l-i}\zeta^{i-(s-m)} \left((1+\zeta)^{r-s}-1\right)&\equiv & 0\bmod p^{t+1}\\
 \implies \underset{\substack{i\leq j\leq r-l\\j\equiv (s-m)\bmod (p-1)}}\sum{{r-l-i}\choose j-i}-\underset{\substack{i\leq j\leq s-l\\j\equiv (s-m)\bmod (p-1)}}\sum{{s-l-i}\choose j-i}&\equiv & 0\bmod p^{t+1}.\\
\end{eqnarray*}
\textbf{Claim:}
$S_{r,i,l,m}\equiv \begin{cases}
                                \underset{\substack{i\leq j<s-m\\ j\equiv (s-m)\bmod (p-1)}}\sum {{r-l}\choose i}\left({{r-l-i}\choose j-i}-{{s-l-i}\choose j-i}\right)\bmod p^t\quad \text{if}\quad i<s-m, \ 0\leq l\leq c\\
                                0 \bmod p^t\quad\text{if} \quad i = s-m,\ l\leq m\\
                                {{r-l}\choose r-m}{{r-m}\choose i}\bmod p^t\quad\text{if}\quad i>s-m,\ l\leq m.
                                  \end{cases}$

We will prove above claim in two cases, $l\leq m$ and $l>m$.\\
\textbf{Case (i)} $0\leq l\leq m$\\
Observe that $r-m+p-1-(r-l) = l+p-1-m\geq 0$ and  $s-m+p-1-(s-l) = l+p-1-m\geq 0$ this gives \\
 \begin{eqnarray*}
 \underset{\substack{r-m\leq j\leq r-l\\j\equiv (s-m)\bmod (p-1)}}\sum{{r-l-i}\choose j-i} &=& \begin{cases}
                        {{r-l-i}\choose r-m-i}+{{r-l-i}\choose r-m+p-1-i}\quad \text{if}\quad l+p-1-m =0\\
                         {{r-l-i}\choose r-m-i}\quad \text{if}\quad l+p-1-m >0
                         \end{cases}\\
                         &=& \begin{cases}
                        {{r-l-i}\choose r-m-i}+1\quad \text{if}\quad l+p-1-m =0\\
                         {{r-l-i}\choose r-m-i}\quad \text{if}\quad l+p-1-m >0
                         \end{cases}\\
\underset{\substack{s-m\leq j\leq s-l,\ i\leq j\\j\equiv (s-m)\bmod (p-1)}}\sum{{s-l-i}\choose j-i} &=& \begin{cases}
                        {{s-l-i}\choose s-m-i}+{{s-l-i}\choose s-m+p-1-i}\quad \text{if}\quad l+p-1-m =0,\ 0\leq i\leq s-m\\
                         {{s-l-i}\choose s-m-i}\quad \text{if}\quad l+p-1-m >0,\ 0\leq i\leq s-m\\
                        {{s-l-i}\choose s-m+p-1-i}\quad \text{if}\quad l+p-1-m =0,\ s-m< i\leq s-l\\
                        0 \quad \text{if}\quad l+p-1-m >0,\ s-m< i\leq s-l\\
                         \end{cases}\\
                &=& \begin{cases}
                        {{s-l-i}\choose s-m-i}+1\quad \text{if}\quad l+p-1-m =0,\ 0\leq i\leq s-m\\
                         {{s-l-i}\choose s-m-i}\quad \text{if}\quad l+p-1-m >0,\ 0\leq i\leq s-m\\
                       1  \quad \text{if}\quad l+p-1-m =0,\ s-m< i\leq s-l\\
                        0 \quad \text{if}\quad l+p-1-m >0,\ s-m< i\leq s-l.
                         \end{cases}
 \end{eqnarray*}
Now for $0\leq i\leq s-m$ observe that ${{r-l-i}\choose {r-m-i}}\equiv {{s-l-i}\choose s-m-i}\bmod p^t$.  Above computation implies that \\
$ \underset{\substack{r-m\leq j\leq r-l\\j\equiv (s-m)\bmod (p-1)}}\sum{{r-l-i}\choose j-i}-\underset{\substack{s-m\leq j\leq s-l\\j\equiv (s-m)\bmod (p-1)}}\sum{{s-l-i}\choose j-i}\equiv \begin{cases}
                                                   0\bmod p^t\quad\text{if}\quad 0\leq i\leq s-m\\
                                                   {{r-l-i}\choose r-m-i}\quad\text{if}\quad s-m<i\leq s-l.
                                                   \end{cases}$\\
Hence we have 
$$S_{r,i,l,m} \equiv\begin{cases}
                               {{r-l}\choose i}\underset{\substack{ i\leq j<s-m\\ j\equiv (s-m)\bmod (p-1)}}\sum\left({{s-l-i}\choose j-i}-{{r-l-i}\choose j-i}\right)\bmod p^t\quad\text{if}\quad i<s-m\\
                                0\bmod p^t\quad\text{if}\quad i= s-m\\
                               - {{r-l}\choose r-m}{{r-m}\choose i}\bmod p^{t+1}\quad\text{if}\quad s-m<i\leq s-l.
                                \end{cases}$$

\textbf{Case (ii)} $m<l\leq c$\\
 In this case $$\underset{\substack{r-l< j< r-m\\j\equiv (s-m)\bmod (p-1)}}\sum{{r-l-i}\choose j-i} = 0$$
 $$\underset{\substack{s-l< j<s-m,\\j\equiv (s-m)\bmod (p-1)}}\sum{{s-l-i}\choose j-i} = 0.$$
 Since summations are empty because $r-m-(p-1)-(r-l+1) = l-(p-1)-m-1<0$ and $s-m-(p-1)-(s-l-1) = l-(p-1)-m-1< 0$.
 $$S_{r,i,l,m} \equiv
                               {{r-l}\choose i}\underset{\substack{ i\leq j<s-m\\ j\equiv (s-m)\bmod (p-1)}}\sum\left({{s-l-i}\choose j-i}-{{r-l-i}\choose j-i}\right)\bmod p^{t+1}.$$
Hence we have proved our claim and so first part of our Lemma is done. \\
Now we will prove second part of our Lemma. \\
\textbf{Case (i)}. $c = 0$\\
For $0\leq i<s-m$,  we have $j<s-m \leq b-m\leq p$ this gives $j-i<p$ implies $\nu((j-i)!) = 0$ therefore ${{s-l-i}\choose j-i}-{{r-l-i}\choose j-i} = 0\bmod p^t$.  This gives our result for $0\leq i<s-m$ and for $i = s-m$ is true by part first.\\
\textbf{Case (ii)} $c\geq 1\ \& \ 0\leq i<s-m$\\
Note that $$\nu\left({{s-l-i}\choose j-i}-{{r-l-i}\choose j-i}\right)\geq t-\nu((j-i)!)$$
$$\& \ \ j-i\leq j\leq s-m-(p-1)\leq b+1-(c+m)+(c-1)p$$
here $c-1\leq p-1$ and $b-m-c+1\leq p-1$ if either $b\leq p-1$ or $c+m\geq 2$. So $\nu((j-i)!)\leq \nu((p-1+(c-1)p)!)\leq c-1\ \ \ \implies \ \ t-\nu((j-i)!)\geq t+1-c$. Therefore $S_{r,i,l,m}\equiv\ 0\bmod p^{t+1-c}$, in case either $2\leq b\leq p-1$ or $b =p, c+m\geq 2$.\\
Now if $b = p$ and $c+m<2$ as $c\geq 1$ then we have $c = 1\  \&\ m = 0$ so, 
$$j-i\leq 1-c-m+cp\leq cp\ \ \ \implies \ \ \ \nu((j-i)!)\leq \nu((cp)!)\leq c$$
$$\hspace{8em}\implies \ \ \ \ \ t-\nu((j-i)!)\geq t-c$$
$S_{r,i,l,m}\equiv\ 0\bmod p^{t-c}$, in case $b=p,c+m<2$.\\
For $i = s-m$,  we have $S_{r,i,l,m}\equiv 0\bmod p^t$ and so is zero mod $p^{t-c}$ or mod $p^{t-(c-1)}$ as $c\geq 1$. 
\end{proof}

\begin{lemma}{\label{coeff 5.1}}
Let $r=b+c(p-1)+p^t(p-1)d$ and suppose $d\geq 0,\ t\geq 2,\ 2\leq b\leq p,\  0\leq m\leq c-1\leq p-2$.  Then for $0\leq j,  l\leq c-1$ we have modulo $p$\\
${{r-l}\choose b-m+j(p-1)}\equiv \begin{cases}
                                                        {{b-c-l}\choose b-m-j}{c\choose j}&\text{if}\quad 0\leq j\leq b-m,\ 0\leq l \leq b-c\\
                                                        {{p+b-c-l}\choose b-m-j}{{c-1}\choose j}&\text{if}\quad 0\leq j\leq b-m,\ b-c+1\leq l \leq b-c+p\\
                                                         {{2p+b-c-l}\choose b-m-j}{{c-2}\choose j} &\text{if} \quad 0\leq j\leq b-m,\ b-c+p+1\leq l\leq b-c+2p\\
                                                        {{b-c-l}\choose p+b-m-j}{c\choose j-1}&\text{if}\quad b-m+1\leq j\leq b-m+p,\ 0\leq l \leq b-c\\
                                                       {{p+b-c-l}\choose p+b-m-j}{{c-1}\choose j-1} &\text{if}\quad b-m+1\leq j\leq b-m+p,\ b-c+1\leq l\leq b-c+p\\
                                                       {{2p+b-c-l}\choose p+b-m-j}{{c-2}\choose j-1} &\text{if} \quad b-m+1\leq j\leq b-m+p,\ b-c+p+1\leq l\leq b-c+2p\\
                                                        {{p+b-c-l}\choose 2p+b-m-j}{{c-1}\choose j-2} &\text{if} \quad b-m+p+1\leq j\leq b-m+2p,\ b-c+1\leq l\leq b-c+p\\
                                                         {{2p+b-c-l}\choose 2p+b-m-j}{{c-2}\choose j-2} &\text{if} \quad b-m+p+1\leq j\leq b-m+2p,\ b-c+p+1\leq l\leq b-c+2p.
                                                       \end{cases}$
\end{lemma}
\begin{proof}
The proof is a straightforward application of Lucas' Theorem (Theorem $2.4$, \cite{BG}).
\end{proof}

\begin{lemma}{\label{lmk68}}
Let $r = b+c(p-1)+p^t(p-1)d$ and suppose $d\geq 0,\ t\geq 2,\ 2\leq b\leq p, \ 1\leq c\leq p-2$. Also assume that $0\leq m\leq p-1$ and $(b, m)\not = (p,0)$.
\begin{enumerate}
\item If $0\leq m\leq l\leq b-c$ and $0\leq j\leq c-1$ then
$$\frac{{{r-l}\choose b-m+j(p-1)}}{p}\equiv (-1)^{l-m}\frac{{{b-m}\choose j}{{p-1+m-l}\choose c-1-j}}{{{b-m-c}\choose l-m}{{b-m}\choose c}}\bmod p.$$
\item If $b\leq m\leq l\leq p+b-c$ and $1\leq j\leq c-1$ then
$$\frac{{{r-l}\choose b-m+j(p-1)}}{p}\equiv (-1)^{l-m}\frac{{{p+b-m-1}\choose j-1}{{p-1+m-l}\choose c-1-j}}{{{p+b-m-c}\choose l-m}{{p+b-m-1}\choose c-1}}\bmod p.$$
\end{enumerate}
\end{lemma}
\begin{proof}
Let $A = \underset{0\leq i\leq n}\sum a_ip^i, \ B = \underset{0\leq i\leq n}\sum b_ip^i$ and $A-B = \underset{0\leq i\leq n}\sum c_ip^i$ are in $p$-adic expansion. If $p^e||{A\choose B}$ then by \cite{k68}
\begin{equation}{\label{eqn k68}}
{A\choose B}\equiv (-p)^e\underset{0\leq i\leq n}\Pi\frac{a_i}{b_ic_i}\bmod p^{e+1}.
\end{equation}
We will apply this result for $A = r-l$ and $B = b-m+j(p-1)$ in following cases.\\
\begin{enumerate}
\item In this case observe that following are in $p$-adic expansion
\begin{eqnarray*}
r-l &=& b-c-l+cp+p^t(p-1)d\\
b-m+j(p-1)&=& b-m-j+jp\\
r-l-(b-m+j(p-1)) &=& p-c+j+m-l+(c-j-1)p+p^t(p-1)d.
\end{eqnarray*}
This follows from $0\leq j\leq c\leq b-m\leq p-1$ as $(b,m)\not=(p, 0)$ (for second line) and $0\leq p-b+m+1\leq p-c+j+m-l\leq p-1$ (for last line). Here one proves that $e = 1$, and so by \ref{eqn k68} we have 
\begin{eqnarray*}
\frac{{{r-l}\choose b-m+j(p-1)}}{p}&\equiv &(-1)\frac{c!(b-c-l)!}{j!(b-m-j)!(p-c+j+m-l)!(c-1-j)!}\bmod p\\
&\equiv & (-1)^{l-m}\frac{{{b-m}\choose j}{{p-1+m-l}\choose c-1-j}}{{{b-m-c}\choose l-m}{{b-m}\choose c}}\bmod p.
\end{eqnarray*}
\item  In this case observe that following are in $p$-adic expansion
\begin{eqnarray*}
r-l &=& p+b-c-l+(c-1)p+p^t(p-1)d\\
b-m+j(p-1)&=& p+b-m-j+(j-1)p\\
r-l-(b-m+j(p-1)) &=& p-c+j+m-l+(c-j-1)p+p^t(p-1)d.
\end{eqnarray*}
This follows from $0\leq j\leq c\leq p+b-m$ (for second line) and $0\leq m+1-b\leq p-c+j+m-l\leq p-1$ (for last line). Here again we note that $e = 1$. Then by \ref{eqn k68} we have 
\begin{eqnarray*}
\frac{{{r-l}\choose b-m+j(p-1)}}{p}&\equiv &(-1)\frac{(c-1)!(p+b-c-l)!}{(j-1)!(p+b-m-j)!(p-c+j+m-l)!(c-1-j)!}\bmod p\\
&\equiv & (-1)^{l-m}\frac{{{p+b-m-1}\choose j-1}{{p-1+m-l}\choose c-1-j}}{{{p+b-m-c}\choose l-m}{{p+b-m-1}\choose c-1}}\bmod p.
\end{eqnarray*}
\end{enumerate}
\end{proof}

\begin{lemma}{\label{rk3.15}}
Let $r=s+p^t(p-1)d,\ t\geq 2, \ s = b+c(p-1)\geq m,\ 2\leq b\leq p,\  0\leq c,\  m\leq p-1$ and $0\leq l\leq m$ then $\nu\left({{r-l}\choose r-m}\right) = \nu\left({{r-l}\choose s-m}\right)$.  Furthermore,
\[ \nu\left({{r-l}\choose r-m}\right) =
\begin{cases}
0 & \ \text{if}\quad (b, c) =(p, 0), m=0\\
1  & \ \text{if}\quad (b, c) =(p, 0),l =0,\ m\not = 0\\
0 & \ \text{if}\quad (b, c) =(p, 0),l \not=0,\ m\not = 0\\
0  & \ \text{if}\quad 0\leq m\leq b-c,\ (b,c)\not=(p,0)\\
1  & \ \text{if} \quad b-c+1\leq m\leq b-c+p,\ 0\leq l\leq b-c \\
0 & \ \text{if}\quad b-c+1\leq m\leq b-c+p,\ b-c+1\leq l\leq b-c+p\\
1  & \ \text{if} \quad b-c+p+1\leq m\leq b-c+2p,\ b-c+1\leq l\leq b-c+p\\
0 & \ \text{if}\quad b-c+p+1\leq m\leq b-c+2p,\ b-c+p+1\leq l\leq b-c+2p.
\end{cases}
\] 
\textbf{(A)} Further we assume $0\leq b-m\leq c$, then we have 
\begin{eqnarray*}
 \nu\left({{r-l}\choose b-m}\right) &=&
 \begin{cases}
 0\quad \text{if}\quad 0\leq l\leq m-c\\
 1\quad\text{if}\quad m-c+1\leq l\leq b-c\\
 0\quad\text{if}\quad b-c+1\leq l\leq b-c+p
\end{cases}\\
\nu\left(\frac{p^{2m-b}{{r-l}\choose b-m}}{{{r-l}\choose r-m}}\right)&=&
\begin{cases}
2m-b\quad\text{if}\quad m = b-c,\ 0\leq l\leq m-c\\
2m-b-1 \quad\text{if}\quad m\geq b-c+1,\ 0\leq l\leq m-c\\
2m-b+1\quad\text{if}\quad m = b-c,\ m-c+1\leq l\leq b-c\\
2m-b\quad\text{if}\quad m\geq  b-c+1,\ m-c+1\leq l\leq b-c\\
2m-b\quad\text{if}\quad m\geq b-c+1, \ b-c+1 \leq l\leq b-c+p.
\end{cases}
\end{eqnarray*}
\textbf{(B)}  Further we assume $0\leq b-m+p-1\leq c$, $m<p-1$ and let $c_0 = \frac{p^{2m-b-(p-1)}{{r-l}\choose b-m+p-1}}{{{r-l}\choose r-m}}$ then\linebreak we have 
\begin{eqnarray*}
 \nu\left({{r-l}\choose b-m+p-1}\right) &=&
 \begin{cases}
 0\quad \text{if}\quad 0\leq l\leq m-c+1\\
 1\quad\text{if}\quad m-c+2\leq l\leq b-c+p\\
 0\quad\text{if}\quad b-c+p+1\leq l\leq b-c+2p
\end{cases}\\
\end{eqnarray*}  
$\nu(c_0)=
\begin{cases}
2m-b-(p-1)\ \text{if}\ b-c+p-1\leq m\leq b-c+p,\ 0\leq l\leq m-c+1\\
2m-b-(p-1)-1\ \text{if}\ b-c+p+1\leq m\leq b-c+2p,\ 0\leq l\leq m-c+1\\
2m-b-(p-1)+1\ \text{if}\ b-c+p-1\leq m\leq b-c+p,\ m-c+2\leq l\leq b-c+p\\
2m-b-(p-1)\ \text{if}\ b-c+p+1\leq m\leq b-c+2p,\ m-c+2\leq l\leq b-c+p\\
2m-b-(p-1)\ \text{if}\ b-c+p+1\leq m\leq b-c+2p, \ b-c+p+1 \leq l\leq b-c+2p.
\end{cases}$                                            
\end{lemma}
\begin{proof}
The proof is a straightforward application of the following observations. For $n \in \mathbb{N}$ with $p$-adic expansion $ n = \sum \limits_{i= 0}^{a} n_{i} p^{i}$ we have: $  \nu (n!) = ({n - \sum \limits_{i= 0}^{a} n_{i}}) / ({p-1})$ where $ 0 \leq n_i \leq p-1$. Therefore, $  \nu (n!) = n_{1} + \nu (m!)$ where $m = \sum \limits_{i= 2}^{a} n_{i} p^{i}$.
\end{proof}

\begin{lemma}{\label{invmt 2}}
Let $b,m,c\in\mathbb{N}\cup\{0\}$ such that $m\leq b-c$ then the matrix $B = (b_{j,i})_{\substack{0\leq j\leq c\\0\leq i\leq c}}$  is invertible mod $p$ where $b_{j,i} = {{b-m-c+1+i}\choose b-m-j}$.
\end{lemma}
\begin{proof}
Apply Vandermonde's identity to get
   $b_{j,i} = \sum_{0\leq l\leq c}{{b-m-c+1}\choose {b-m-j-l}}{i\choose l}$. Hence, we can write $B = B'B''$,  where $B'$ and $B''$ has the matrix entries given by ${{b-m-c+1}\choose b-m-j-l}$ and ${i\choose l}$ respectively.  Observe $B''$ is invertible as it is lower triangular with $1$ on the diagonal, so enough to prove $B'$ is invertible. And this we will show by showing $B'$ is full rank. \\
Now Let $X = (x_c,x_{c-1},...,x_0)^t$ such that $BX = 0$. So we get following system of equations                                     
\begin{equation}{\label{beta 1}}
x_j+\sum_{c-j+1\leq l \leq c}b'_{j-1,l}x_{c-l} = 0\ \forall\ \ 1\leq j\leq c
\end{equation}
\begin{equation}{\label{beta 2}}
\sum_{0\leq l\leq c}{{b-m-c+1}\choose {b-m-c-l}}x_{c-l} = 0.
\end{equation}
Now by equation \eqref{beta 1} using induction on $j$ we have $x_j= \beta_jx_0$ where 
\[
\beta_j =\begin{cases}
                   1   & \quad  for\  j= 0 \\ 
                   -{{b-m-c+1}\choose 1}  &\quad for\  j =1\\
                 -\sum_{c-j+1\leq l\leq c}{{b-m-c+1}\choose {b-m-(j+l-1)}}\beta_{c-l} &\quad for\ 2\leq j\leq c.
                \end{cases} 
\]
\textbf{Claim} $\beta_j = (-1)^j{{b-m-c+j}\choose j}$ for all $0\leq j\leq c$.\\
We will prove claim by induction on $j$.  For $j = 0$ it is trivially true. By induction assume for $0\leq j\leq k$ and we will prove it for $j = k+1$.  So we need to prove  $$\beta_{k+1} = (-1)^{k+1}{{b-m-c+k+1}\choose k+1}.$$
By using definition of $\beta_j$ and the induction step,  it is equivalent to prove the following
\begin{eqnarray*}
 -\sum_{c-k\leq l\leq c}(-1)^{c-l}{{b-m-c+1}\choose b-m-(k+l)}{{b-m-c+c-l}\choose c-l} &=&(-1)^{k+1}{{b-m-c+k+1}\choose k+1}.
\end{eqnarray*}
By putting $i = k+1-c+l$, we get the equivalent statement
\begin{align*}
& -\sum_{1\leq i\leq k+1}(-1)^{k+1-i}{{b-m-c+1}\choose b-m-c+1-i}{{b-m-c+k+1-i}\choose k+1-i} =(-1)^{k+1}{{b-m-c+k+1}\choose k+1} \\
\iff &  \ \ \  \sum_{0\leq i\leq k+1}(-1)^i{{b-m-c+1}\choose i}{{b-m-c+k+1-i}\choose b-m-c} = 0. 
\end{align*}
But the last equality follows by Lemma \ref{cmbi 4}.  Now using above claim and equation \eqref{beta 2}, we get 
\begin{align*}
\sum_{0\leq l\leq c}(-1)^{c-l}{{b-m-c+1}\choose {b-m-c-l}}{{b-m-l}\choose c-l}x_0 = 0.
\end{align*}
By Lemma \ref{cmbi 4}, we deduce $X = \textbf{0}\in \mathbb{F}^{c+1}_p$ since ${{b-m++1}\choose b-m-c}\not\equiv 0\bmod p$.
\end{proof}

\begin{lemma}{\label{invmt 1}}
Let $m,n\in \mathbb{N}$ such that $c\leq m$ then $B= \left({{m-c+j}\choose i}\right)_{\substack{1\leq j\leq c\\0\leq i\leq c-1}}\in\mathrm{GL}_c(\mathbb{F}_p)$. 
\end{lemma}
\begin{proof}
Using Vondermond 's identity  for $1\leq j\leq c$, we get
$${{m-c+j}\choose i} = \sum_{0\leq l\leq c-1}{j\choose l}{{m-c}\choose i-l}.$$
Using the above identity, we can write $B = B'B''$ where $B' =\left(b'_{j,l}\right), b'_{j,l} = {j\choose l}$ for $1\leq j\leq c, 0\leq l \leq c-1$ and $B'' = \left({{m-c}\choose i-l}\right)$.  Using the above identity, we can write $B = B'B''$ where $B' =\left(b'_{j,l}\right), b'_{j,l} = {j\choose l}$ for $1\leq j\leq c, 0\leq l \leq c-1$ and $B'' = \left({{m-c}\choose i-l}\right)$. 
Note that $B''$ is an upper triangle with $1$ on the diagonal, so is invertible. Hence to prove $B$ is invertible enough to prove $B'$ is invertible, and this we will prove by proving it is full rank. Take $X^t = (x_0,x_1,...,x_{c-1})\in\mathbb{Z}^{c}_p$ is solution of $B'X =0$.Note that $B''$ is upper triangle with $1$ on diagonal, so is invertible. Hence to prove $B$ is invertible enough to prove $B'$ is invertible, and this we will prove by proving it is full rank. Take $X^t = (x_0,x_1,...,x_{c-1})\in\mathbb{Z}^{c}_p$ is solution of $B'X =0$.
\begin{equation}{\label{j=0}}
\implies \ \ \ \ \ \ \ \ \ \ \ \sum_{0\leq l\leq c-1}{c\choose l}x_l = 0 \quad\text{for }\quad j =c \hspace{10em}
\end{equation}

\begin{equation}{\label{jn0}}
 \sum_{0\leq l\leq j} {j\choose l}x_l= 0\  \forall \ \ 1\leq j\leq c-1.
\end{equation}
Using above system of equation \eqref{jn0},  by induction we will prove $x_l = (-1)^lx_0$ for $0\leq l\leq c-1$. Our claim fallow for $l = 1$ by putting $j = 1$ in system of equation \ref{jn0}. Assume by induction  $x_l = (-1)^lx_0$ for $0\leq l\leq k-1$, and we will prove for $l = k\leq c-1$. Now using $k^{th}$ equation in \eqref{jn0} we get
    $$x_k+\sum_{0\leq l\leq k-1} {k\choose l}x_l  = 0$$
  $$\implies \ \ \ \ \ x_k+\sum_{0\leq l\leq k-1}(-1)^l {k\choose l}  x_0 = 0$$ 
  which gives $-(-1)^kx_0+x_k=0 \ \ \ \implies x_k = (-1)^kx_0 $.  By putting $x_l = (-1)^lx_0$ in equation \eqref{j=0}  to see $x_0 = 0$. Therefore, $B'$ is of full rank.
\end{proof}

For every $n\in\mathbb{Z}^{\geq 0}$ define the function $H(n)$ as follows 
$$H(n) := \prod\limits_{i = 0}^{n-1} i!.$$
From the above definition, it is clear that $H(n)\not\equiv 0\bmod p$ for all $n\leq p$.

\begin{lemma}[D. Grinberg,  P.A. MacMahon]{\label{grinberg}}
For every $a,b,c\in\mathbb{Z}^{\geq 0}$, we have 
\begin{eqnarray*}
\det\left(\left({{a+b+i-1}\choose a+i-j}\right)_{1\leq i,j\leq c}\right)
&=& \det\left(\left({{a+b}\choose a+i-j}\right)_{1\leq i,j\leq c}\right)\\
&=& \frac{H(a)H(b)H(c)H(a+b+c)}{H(b+c)H(c+a)H(a+b)}. 
\end{eqnarray*}
\end{lemma}
For the proof of this lemma see Theorem $8$, \cite{grinberg}.

\section{Towards elimination of JH factors}

\begin{prop}{\label{gen 1}}
 Let $r= s+p^t(p-1)d$, with $p\nmid d$, $s =b+c(p-1)$ and suppose also that $2\leq b\leq p$ and $0\leq m<c\leq\nu(a_p)<p-1$. Further we assume  $t>\nu(a_p)+c-1$ if $(b,  c,  m) \not = (p,  1,  0)$ and $t>\nu(a_p)+c$ if $(b,  c,  m) = (p,  1,  0)$. Then for all $g\in G$ and for $0\leq l\leq c-1$, there exists $f^l\in \indkg(\symqp)$ such that 
\begin{equation}{\label{1st gen 1}}
(T-a_p)f^l\equiv \left[g, \underset{\substack{0<j<s-m\\j\equiv (s-m)\bmod (p-1)}}\sum{{r-l}\choose j}x^{r-j}y^j\right].
\end{equation}
Further suppose $(b,c, m)\not=(p, 1,0),\ \nu(a_p)>c$, $t>\nu(a_p)+c$, and assume $0\leq l\leq b-c$ if $m = 0$. Then  for all $g\in G$ there exists $f^l\in \indkg(\symqp)$ such that   
\begin{equation}{\label{2nd gen 1}}
(T-a_p)\left(\frac{f^l}{p}\right)\equiv \left[g, \underset{\substack{0<j<s-m\\j\equiv (s-m)\bmod (p-1)}}\sum\frac{{{r-l}\choose j}}{p}x^{r-j}y^j\right].
\end{equation}
\end{prop}
\begin{proof}
We begin by observing that the coefficients $\frac{{{r-l}\choose j}}{p}$ in (\ref{2nd gen 1}) are integral if $0\leq m\leq l\leq b-c$ or $b\leq m\leq l\leq b-c+p$.
Consider the following functions
\begin{eqnarray*}
f_{3,l} &=& \sum_{\lambda\in I^*_1}\left[ g^0_{2,p\lambda}, \frac{F_l(x,y)}{\lambda^{m-l}p^l(p-1)}\right]\\
f_{2,l} &=& \left[ g^0_{2,0},{{r-l}\choose r-m}\frac{F_m(x,y)}{p^m}\right]\\
f_{1,l} &=&\left[g^0_{1,0}, \frac{1}{a_p}\underset{\substack{s-m\leq j<r-m\\j\equiv (r-m)\bmod (p-1)}}\sum{{r-l}\choose j}x^{r-j}y^j\right]
\end{eqnarray*}
\[  f_0 = \begin{cases}
              [1,\ F_s(x,y)] & \text{if}\ r\equiv m\bmod (p-1)\\
              0   \   & \    \text{else}
              \end{cases}
\]
\begin{align*}
T^+\left(\left[ g^0_{2,p\lambda}, \frac{F_l(x,y)}{\lambda^{m-l}p^l(p-1)}\right]\right) = \sum_{\mu\in I^*_1}\left[ g^0_{3,p\lambda+p^2\mu}, \sum_{0\leq j\leq s-l}\frac{p^{j-l}(-\mu)^{s-l-j}}{\lambda^{m-l(p-1)}}\left({{r-l}\choose j}- {{s-l}\choose j}\right)x^{r-j}y^j\right]\\
+\sum_{\mu\in I}\left[ g^0_{3,p\lambda+p^2\mu}, \sum_{s-l+1\leq j\leq r-l}\frac{p^{j-l}(-\mu)^{r-l-j}}{\lambda^{m-l}(p-1)}{{r-l}\choose j}x^{r-j}y^j\right]\\
\hspace{5em} -\left[ g^0_{3,p\lambda}, \frac{p^{s-2l}}{\lambda^{m-l}(p-1)}x^{r-s+l}y^{s-l}\right].
\end{align*}
\vspace{2mm}\\
Now we will estimate the valuation of the coefficients in the three sums (I),(II) and (III) above.  In (I),  for $j\geq 1$,  $\nu\left({{r-l}\choose j}-{{s-l}\choose j}\right)\geq t-\nu(j!) \implies$ $ j-l+t-\nu(j!)\geq t-(c-1)+1\geq \nu(a_p)+1>1$.  For (III),  $s-2l\geq b+c(p-1)-2(c-1)\geq b+c(p-3)+2\geq b+2\geq 4$. For (II) the same computation as in (III) shows that $j-l\geq b+3\geq 5$.  Therefore we have $T^+\left(\frac{f_{3,l}}{p}\right)\equiv\ 0\bmod p$.  Now, 
\begin{align*}
T^-\left(\left[ g^0_{2,p\lambda}, \frac{F_l(x,y)}{\lambda^{m-l}p^l(p-1)}\right]\right) =- \left[ g^0_{1,0}, \sum_{0\leq j\leq s-l}\frac{p^{r-s}\lambda^{s-m-j}}{(p-1)}{{s-l}\choose j}x^{r-j}y^j\right]\\
+\left[ g^0_{1,0}, \sum_{0\leq j\leq r-l}\frac{\lambda^{r-m-j}}{(p-1)}{{r-l}\choose j}x^{r-j}y^j\right].
\end{align*}
In the first sum the valuation of the coefficients are atleast $r-s\gg 0$, and therefore we have 
\begin{eqnarray*}
  T^-\left(\frac{ f_{3,l}}{p}\right) &\equiv & \left[g^0_{1,0}, \underset{\substack{0\leq j\leq r-m\\ j\equiv (r-m)\bmod (p-1)}}\sum\frac{{{r-l}\choose j}}{p}x^{r-j}y^j\right].\\
\end{eqnarray*}
For $f_{2,l}$ we observe that similar computation as above gives $T^{+} \left(\frac{f_{2, l}}{p}\right) \equiv 0 \bmod p$ and 
$$T^-\left(\frac{f_{2,l}}{p}\right)\equiv \left[g^0_{1,0}, \ \frac{{{r-l}\choose r-m}}{p}x^my^{r-m}\right].$$
Now, 
$$ T^+(f_{1,l}) = \sum_{\lambda\in I^*_1}\left[ g^0_{2,p\lambda}, \sum_{0\leq j\leq r}\frac{p^j(-\lambda)^{s-m-j}}{a_p}\underset{\substack{s-m\leq i<r-m\\i\equiv (r-m)\bmod (p-1)}}\sum{{r-l}\choose i}{i\choose j}x^{r-j}y^j\right]\\$$
\begin{equation}{\label{tf1}}
\hspace{19em} +\left[ g^0_{2,0}, \underset{\substack{s-m\leq j<r-m\\j\equiv (r-m)\bmod (p-1)}}\sum\frac{p^j}{a_p}{{r-l}\choose j}x^{r-j}y^j\right].
\end{equation}

Here we note $m\leq c-1$,  and that for $j\geq s-(c-1)$, $j-\nu(a_p)\geq b+(c-1)(p-1)-(c-1)+p-1-\nu(a_p)>b+(c-1)(p-2)\geq 2$ (as $c\geq 1$).  Thus the first sum truncates to $j\leq s-c$ and the second sum is zero mod $p$.
\begin{eqnarray*}
T^+(f_{1,l}) &=& \sum_{\lambda\in I^*_1}\left[ g^0_{2,p\lambda}, \sum_{0\leq j\leq s-c}\frac{p^j(-\lambda)^{s-m-j}}{a_p}\underset{\substack{s-m\leq i<r-m\\i\equiv (r-m)\bmod (p-1)}} \sum{{r-l}\choose i}{i\choose j}x^{r-j}y^j\right]\\
\implies T^+(f_{1,l}) &=& \sum_{\lambda\in I^*_1}\left[ g^0_{2,p\lambda}, \sum_{0\leq j\leq s-c}\frac{p^j(-\lambda)^{s-m-j}}{a_p}S_{r,j,l,m}\ x^{r-j}y^j\right]
\end{eqnarray*}
where $S_{r,j,l,m}$ is defined in equation \eqref{dfnsrjm}.  Now $(b,c,m)\not=(p,1,0)$ implies that either $b\leq p-1$ or $c+m\geq 2$ (or both).  So Lemma \ref{srjm} gives $\nu(S_{r,j,l,m})\geq t+1-c$,  and therefore the valuation of above coefficients is atleast $j+t+1-c-\nu(a_p)\geq t-(\nu(a_p)+c-1)>0$ giving $ T^+(f_{1,l})\equiv\ 0\bmod p$. For $(b,c,m)= (p,1,0)$, Lemma \ref{srjm} gives $\nu(S_{r,j,l,m})\geq t-c$, so the valuation of above coefficients is atleast $j+t-c-\nu(a_p)\geq t-(c+\nu(a_p))>0$, again giving $T^+(f_{1,l})\equiv\ 0\bmod p$. Observe that the same calculation gives $T^+\left(\frac{f_{1,l}}{p}\right)\equiv\ 0\bmod\ p$ when $(b,c,m)\not= (p,1,0)$ (using $t>\nu(a_p)+c$). Next,
$$T^-(f_{1,l}) = \left[ 1, \underset{\substack{s-m\leq j<r-m\\ j\equiv (r-m)\bmod (p-1)}}\sum\frac{p^{r-j}}{a_p}{{r-l}\choose j}x^{r-j}y^j\right]$$
where the valuation of the coefficients is atleast $r-j-\nu(a_p)\geq m+p-1-\nu(a_p)>0\ \ \Rightarrow\ \ T^-(f_{1,l})\equiv\ 0\bmod\ p$.  Also, 
$$T^-\left(\frac{f_{1,l}}{p}\right) = \left[ 1, \underset{\substack{s-m\leq j<r-m\\ j\equiv (r-m)\bmod (p-1)}}\sum\frac{p^{r-j-1}}{a_p}{{r-l}\choose j}x^{r-j}y^j\right].$$
 For $j\leq r-m-2(p-1)$ the valuation of the above coefficients are atleast $r-j-1-\nu(a_p)\geq p-2+m+p-1-\nu(a_p)>0$.  For $j= r-m-(p-1)$, the valuation of the coefficient is
$r-j-1-\nu(a_p)+\nu({{r-l}\choose r-m-(p-1)})\geq m+\nu\left({{r-l}\choose r-m-(p-1)}\right)-1+p-1-\nu(a_p)>m+\nu\left({{r-l}\choose r-m-(p-1)}\right)-1\geq 0$. Observe that the last inequality is clear if $m\geq 1$.  Further if $ m = 0$ then $\nu{{r-l}\choose p-1-l}\geq 1$ since $b-c-l<p-1-l$ as $b-c<p-1$ (note: $(b,c,m)\not= (p,1,0)$ and $c\geq 1$). Therefore we have $T^-\left(\frac{f_{1,l}}{p}\right) \equiv 0\bmod p$ when $(b,c,m)\not= (p,1,0)$.\\
If $r\equiv m\bmod (p-1)$ then 
\begin{eqnarray*}
T^+(f_0) &=& \sum_{\lambda\in I^*_1}\left[ g^0_{1,\lambda}, (-1+(-\lambda)^{r-s})x^r\right]+\sum_{\lambda\in I_1}\left[g^0_{1,\lambda},\sum_{1\leq j\leq r-s}p^j{{r-s}\choose j}(-\lambda)^{r-s-j}x^{r-j}y^j\right]\\
&&\hspace*{30em}+[g^0_{1,0},-x^r]\\
\Rightarrow \ \ \ T^+(f_0) &\equiv & -[g^0_{1,0},\  x^r]\\
T^-(f_0) &=& \left[ \alpha, \ -p^rx^r + p^sx^sy^{r-s}\right]\equiv\ 0\bmod p\\
T^+\left(\frac{f_0}{p}\right) &=& \sum_{\lambda\in I^*_1}\left[ g^0_{1,\lambda}, \frac{(-1+(-\lambda)^{r-s})}{p}x^r\right] +\sum_{\lambda\in I_1}\left[g^0_{1,\lambda},\sum_{1\leq j\leq r-s}p^{j-1}{{r-s}\choose j}(-\lambda)^{r-s-j}x^{r-j}y^j\right]\\
&&\hspace{30em}+[g^0_{1,0},\ -\frac{1}{p}x^r].
\end{eqnarray*}
Observe that if $j = 1 $, $\ {{r-s}\choose j} = r-s$ which is divisible  by $p^t,\ t\geq 1$. Thus
\begin{eqnarray*}
 \ \ \ T^+\left(\frac{f_0}{p}\right) &\equiv & -[g^0_{1,0},\ \frac{1}{p} x^r]\\
T^-\left(\frac{f_0}{p}\right) &=& \left[ \alpha, \ -p^{r-1}x^r + p^{s-1}x^sy^{r-s}\right]\equiv\ 0\bmod p.
\end{eqnarray*}
Note that $a_pf_{3,l},a_pf_{2,l},a_pf_{0}\equiv  0$, and also $\frac{a_pf_{3,l}}{p},\frac{a_pf_{2,l}}{p},\frac{a_pf_{0}}{p}\equiv 0$ as $\nu(a_p)>c$.
\begin{eqnarray*}
(T-a_p)(f_{3,l})&\equiv &\left[g^0_{1,0}, \underset{\substack{0\leq j\leq r-m\\j\equiv (r-m)\bmod (p-1)}}\sum{{r-l}\choose j}x^{r-j}y^j\right]\\
(T-a_p)(f_{2,l})&\equiv &\left[g^0_{1,0},{{r-l}\choose r-m}x^my^{r-m}\right]
\end{eqnarray*}
\begin{eqnarray*}
(T-a_p)(f_{1,l})&\equiv &-\left[g^0_{1,0}, \underset{\substack{s-m\leq j<r-m\\j\equiv (r-m)\bmod (p-1)}}\sum{{r-l}\choose j}x^{r-j}y^j\right]
\end{eqnarray*}
\[(T-a_p)(f_{0,l})\equiv  \begin{cases}
              -[g^0_{1,0},\  x^r] & \ if\ r\equiv m\bmod (p-1)\\
              0   \   & \    else.
              \end{cases}
\]
Hence  $f^l := f_{3,l}-f_{2,l}+f_{1,l}+f_{0,l}$ gives the required result.
\end{proof}

\begin{prop}{\label{mono 1}}
Let $r= s+p^t(p-1)d$, with $p\nmid d$, $s =b+c(p-1)$ and suppose also that $c\leq b\leq p$ and $1\leq m<c<\nu(a_p)<p-1$. Further if $t>\nu(a_p)+c$  then the monomials $x^{r-b+m-j(p-1)}y^{b-m+j(p-1)}$ are in $\kerp$ for $0\leq j\leq c-1$.
\end{prop}
\begin{proof}
We begin by noting that if $m\leq l\leq b-c$ then the coefficients of (\ref{1st gen 1}), ${{r-l}\choose b-m+j(p-1)}\equiv 0\bmod p$ for all $j$, due to which in some cases the matrices $A$ below are not invertible mod $p$.  So we use (\ref{2nd gen 1}) instead of (\ref{1st gen 1}) in this situation.

\textbf{Case (i)} $b\geq 2c-1$ and $1\leq m\leq c-1$\\
Let us consider the matrix $A = \left(a_{j,l}\right)$ over $\mathbb{Z}_p$ where
\[ a_{j,l} = 
\begin{cases}
{{r-l}\choose b-m+j(p-1)} & \text{if}\quad 0\leq j\leq c-1, \ 0\leq l\leq m-1\vspace*{2mm}\\
\frac{{{r-l}\choose b-m+j(p-1)}}{p} & \text{if}\quad 0\leq j\leq c-1, \ m\leq l\leq c-1.
\end{cases}
\]
Here we note that $m\leq c-1\leq b-c$, and so by Lemma \ref{coeff 5.1} and Lemma \ref{lmk68} we have 
\[ a_{j,l} \equiv
\begin{cases}
{{b-c-l}\choose b-m-j}{c\choose j} & \text{if}\quad 0\leq j\leq c-1, \ 0\leq l\leq m-1\vspace*{2mm}\\
\frac{(-1)^{l-m}{{b-m}\choose j}{{p-1+m-l}\choose c-1-j}}{{{b-c-m}\choose l-m}{{b-m}\choose c}} & \text{if}\quad 0\leq j\leq c-1, \ m\leq l\leq c-1.
\end{cases}
\]
Now let us write $A$ in blocks as follows
\begin{eqnarray}{}\label{Block 1 A}
A = \begin{pmatrix}
 A'& B'\\
 A''&B''
 \end{pmatrix} 
\end{eqnarray}
 where we divide the ranges of $l$ and $j$ into two non empty intervals  $[0,\ m-1],[m,\ c-1]$ and $[0,\ c-m-1], [c-m, \ c-1]$ respectively, determining the order of the blocks. \\
 \textbf{Subcase (i)} $0\leq l\leq m-1$ and $0\leq j\leq c-1$\\
 Here we observe that 
\begin{equation}{\label{binomial A'}}
{{b-c-l}\choose b-m-j}\equiv 0\bmod p\iff j<c-m+l.
\end{equation}
 Thus modulo $p$, ${A'}$ is zero (as $j\leq c-m-1$) and ${A''}$ is lower triangular with non zero diagonal entries given by ${c\choose j}$  (note: $j = c-m+l$ gives the diagonal).  Hence $A''$ is invertible.  \\
 \textbf{Subcase (ii)} $m\leq l\leq c-1$ and $0\leq j\leq c-m-1$\\
 In this case we note that $B'$ is invertible $\bmod\ p$ if and only if 
 $$B_1 = \left({{p-1+m-l}\choose c-1-j}\right)_{\substack{0\leq j\leq c-m-1\\ m\leq l\leq c-1}}$$
 is invertible mod $p$ since in this case ${{b-m}\choose j}, \ {{b-m-c}\choose l-m}$ and ${{b-m}\choose c}$ are all non zero mod $p$. Here we also note that $B_1$ is invertible $\bmod\ p$ if and only if
 $$B'_1 = \left({{p-c+m+l'-1}\choose p-c+l'-j'}\right)_{1\leq j',l'\leq c-m}$$
 is invertible $\bmod\ p$ since $B'_1$ is obtained by first putting $j' = c-m-j,  l' = c-l$ and then using ${{a}\choose b} = {{a}\choose a-b}$.  By Lemma \ref{grinberg}, we have
 $$\det(B'_1) = \frac{H(p-c)H(m)H(c-m)H(p)}{H(p-(c-m))H(c)H(p-m)}\not\equiv 0\bmod p.$$
 Therefore $A$ is invertible over $\mathbb{Z}_p$ as $A'$ is zero mod $p$ and both $A'', \ B'$ are invertible. Now for a fixed $ j''\in [0, \ c-1]$ let $\textbf{d}_{j''} =(d_0,d_1,...,d_{c-1})\in \mathbb{Z}^c_p$ be a vector such that $\textbf{d}_{j''} = A^{-1}\textbf{e}_{j''}$. Then by Proposition \ref{gen 1} we get 
$$(T-a_p)\left(\sum_{0\leq l\leq m-1}d_lf^l+\sum_{m\leq l\leq c-1}d_l\frac{f^l}{p}\right) = [g,\ x^{r-b+m-j''(p-1)}y^{b-m+j''(p-1)}]\bmod p$$
where $f^l$ are from Proposition \ref{gen 1}. 

\textbf{Case (ii)} $c\leq b\leq 2c-2 $ and $1\leq m\leq b-c+1$\\
In this case we consider $A = (a_{j,l})$ over $\mathbb{Z}_p$ where, 
\[ a_{j, l} = 
\begin{cases}
{{r-l}\choose b-m+j(p-1)} &\text{if}\quad 0\leq j\leq c-1, \ 0\leq l\leq m-1\quad \text{or}\quad b-c+1\leq l\leq c-1\vspace*{2mm}\\
\frac{{{r-l}\choose b-m+j(p-1)}}{p} &\text{if}\quad 0\leq j\leq c-1, \ m\leq l\leq b-c
\end{cases}
\]
By using Lemma \ref{coeff 5.1} and Lemma \ref{lmk68} we have 
\[a_{j,l}\equiv 
\begin{cases}
{{b-c-l}\choose b-m-j}{{c}\choose j} &\text{if}\quad 0\leq j\leq c-1, \ 0\leq l\leq m-1\vspace*{2mm}\\
\frac{(-1)^{l-m}{{b-m}\choose j}{{p-1+m-l}\choose c-1-j}}{{{b-m-c}\choose l-m}{{b-m}\choose c}} &\text{if}\quad 0\leq j\leq c-1, \ m\leq l\leq b-c\vspace*{1mm}\\
{{p+b-c-l}\choose b-m-j}{{c-1}\choose j} &\text{if}\quad 0\leq j\leq c-1, \ b-c+1\leq l\leq c-1.
\end{cases}
\]
Here we note that for $b-c+1\leq l\leq c-1$
$${{p+b-c-l}\choose b-m-j}{{c-1}\choose j} = {{b-m}\choose j}{{p-1+m-l}\choose c-1-j}\frac{(p+b-c-l)!(c-1)!}{(b-m)!(p-1+m-l)!}.$$
Now let 
\[\beta_l = 
\begin{cases}
\frac{(-1)^{l-m}}{{{b-m-c}\choose l-m}{{b-m}\choose c}} &\text{if}\quad m\leq l\leq b-c\vspace*{1mm}\\
\frac{(p+b-c-l)!(c-1)!}{(b-m)!(p-1+m-l)!} &\text{if}\quad b-c+1\leq l\leq c-1.
\end{cases}
\]
Therefore we have 
\[a_{j,l}\equiv 
\begin{cases}
{{b-c-l}\choose b-m-j}{{c}\choose j} &\text{if}\quad 0\leq j\leq c-1, \ 0\leq l\leq m-1\\
\beta_l {{b-m}\choose j}{{p-1+m-l}\choose c-1-j}  &\text{if}\quad 0\leq j\leq c-1, \ m\leq l\leq c-1.\\
\end{cases}
\]
Now we write $A$ as in (\ref{Block 1 A}) and observe that similar computation as in Case (i) above gives mod $p$: (i) $A'$ is zero, (ii) $A''$ is invertible and (iii) $B'$ is invertible (as $\beta_{l}$ are all units). Hence $A$ is invertible. 
Now for a fixed $0\leq j'' \leq c-1$ let $\textbf{d}_{j''} =(d_0,d_1,...,d_{c-1})\in \mathbb{Z}^c_p$ be a vector such that $\textbf{d}_{j''} = A^{-1}\textbf{e}_{j''}$. Then taking $f = \left(\sum_{0\leq l\leq c-1}d_l\frac{f^l}{p^\sigma}\right)$ in Proposition \ref{gen 1} we get the desired result where $ \sigma$ is $1$ if $m\leq l\leq b-c$ and $0$ otherwise.

\textbf{Case (iii)} $c\leq b\leq 2c-3$ and $b-c+2\leq m\leq c-1$\\
In this case we consider the following matrix  
$A = \left(a_{j,l}\right)_{0\leq j,l\leq c-1}$ where $a_{j,l} = {{r-l}\choose b-m+j(p-1)}$.\linebreak
By Lemma \ref{coeff 5.1}, we have 
\begin{eqnarray}{\label{bino case iii}}
a_{j,l} \equiv \begin{cases}
{{b-c-l}\choose b-m-j}{c\choose j} &\text{if}\ 0\leq j\leq b-m, \ 0\leq l\leq b-c\\
{{p+b-c-l}\choose b-m-j}{{c-1}\choose j} &\text{if}\ 0\leq j\leq b-m,\ b-c+1\leq l\leq c-1\\
{{b-c-l}\choose p+b-m-j}{c\choose j-1} &\text{if}\ b-m+1\leq j\leq c-1,\ 0\leq l\leq b-c\\
{{p+b-c-l}\choose p+b-m-j}{{c-1}\choose j-1} &\text{if}\ b-m+1\leq j\leq c-1,\ b-c+1\leq l\leq c-1.
\end{cases}
\end{eqnarray}
 Now let us write $A$ in blocks as
$$A = \begin{pmatrix}
 {A'}& {B'}& {C'}\\
 {A''}& {B''}& {C''}\\
 {A'''}&{B'''}&{C'''}
 \end{pmatrix}$$  
 where we divide the ranges of $l$ and $j$ into the non-empty intervals $[0,\ b-c],[b-c+1,\ m-1],[m, \ c-1]$ and $[0,\ c-m-1], [c-m, \ b-m], [b-m+1, \ c-1]$ respectively, which determine the order of the blocks. We now show that $A$ is invertible. Using (\ref{bino case iii}) and similar arguments as that of (\ref{binomial A'}) in Case (i) we deduce that modulo $p$: (i) $A'$,  $A'''$ and $C'''$ are zero and (ii) $A''$, $B'''$ are lower triangular with non-zero entries in the diagonal (hence invertible).  For $C'$ we have $m\leq l\leq c-1$ and $0\leq j\leq c-m-1$.
By Vandermonde's identity 

$$ {{p+b-c-l}\choose b-m-j}{{c-1}\choose j} = \sum_{0\leq l'\leq c-m-1}{{p+b-2c+1}\choose b-m-j-l'}{{c-1}\choose j}{{c-1-l}\choose l'}$$
whence ${C'}$ is a product of two matrices as follows:
$${C'}= \left({{p+b-2c+1}\choose b-m-j-l'}{{c-1}\choose j} \right)_{\substack{0\leq j\leq c-m-1\\ 0\leq l'\leq c-m-1}}\cdot\left({{c-1-l}\choose l'}\right)_{\substack{0\leq l'\leq c-m-1\\ m\leq l\leq c-1}}.$$

Observe that $\det\left({{c-1-l}\choose l'}\right)\not\equiv 0\bmod p$ since this matrix has $1$'s on the off diagonal and $0$'s below it.  Therefore to show that ${C'}$ is invertible, it suffices to show that $\left({{p+b-2c+1}\choose b-m-j-l'} \right)$ is invertible which is equivalent to showing $ \left({{p-c+b-c+1}\choose b-c+1+j'-l''}\right)$ is invertible. The second matrix is obtained by putting $j' = (c-m-j)$ row and $l'' =l'+1$. The latter is invertible mod $p$ by Lemma \ref{grinberg}.  Thus we have
 $$A\equiv\begin{pmatrix}
 \textbf{0}& {B'}& {C'}\\
 {A''}& {B''}& {C''}\\
 \textbf{0}&{B'''}&\textbf{0}
 \end{pmatrix}\bmod p$$  
 where ${A''},{B'''},{C'}$ are of full rank mod $p$, and so $A$ is also of full rank.  Taking $ f = \sum_{0\leq l\leq c-1}d_lf^l$ in Proposition \ref{gen 1} as before we obtain the required result. 
\end{proof}

\begin{prop}{\label{mono 1.2}}
Let $r= s+p^t(p-1)d$ with $p\nmid d$,  and $s =b+c(p-1)$ where $2\leq b\leq c-1\leq p-3$. If $t>\nu(a_p)+c$ and $\ 1\leq m<c<\nu(a_p)<p-1$.
\begin{enumerate}
\item If $1\leq m< b$ then the monomials $x^{r-b+m-j(p-1)}y^{b-m+j(p-1)}$ are in $\kerp$ for $0\leq j\leq b-m$ and $c-m\leq j\leq c-1$.\\
\item If $b\leq m\leq c-1$ then the monomials $x^{r-b+m-j(p-1)}y^{b-m+j(p-1)}$ are in $\kerp$ for $1\leq j\leq c-1$.
\end{enumerate}
\end{prop}
\begin{proof}
We begin by noting that if $b\leq m\leq l\leq p+b-c$ then coefficients of (\ref{1st gen 1}), ${{r-l}\choose b-m+j(p-1)}\equiv 0\bmod p$ for all $j$, due to which in some cases the matrix $A$ below is not invertible mod $p$.  So we use (\ref{2nd gen 1}) instead of (\ref{1st gen 1}) in this situation.

\textbf{Case (i)} $2c-1-p\leq b \leq c-1$ and $ 1 \leq m \leq b-1$\\
Now we consider the matrix $A = \left(a_{j,l}\right)$ over $\mathbb{Z}_p$, where 
\begin{eqnarray}{\label{defn A case 1}}
a_{j,l} = \begin{cases}
             {{r-l}\choose b-m+j(p-1)} & \text{if}\quad 0\leq j\leq b-m ,\ 0\leq l\leq b\\
             {{r-l}\choose b-m+(c-b-1+j)(p-1)} & \text{if}\quad b-m+1\leq j\leq b, \ 0\leq l\leq b.
             \end{cases}
\end{eqnarray}
 Now we write $A$ as a block matrix  $$A= \begin{pmatrix}
 A' & B'\\
 A'' & B''
 \end{pmatrix}$$
 where the ranges of $l$ and $j$ are divided into non-empty intervals $[0,\ m-1], [m, \ b]$ and $[0,\ b-m], \ [b-m+1, \ b]$ respectively,  determining the order of the blocks.  Now we analyse these block matrices in the following subcases.\\
\textbf{Subcase (i)} $A''$ and $B''$:\\
We consider the matrices
 \begin{eqnarray*}
A_1 = \left({{r-l}\choose b-m+j'(p-1)}\right)_{\substack{c-m\leq j'\leq c-1\\ 0\leq l\leq m-1}} \ \text{and} \ B_1 =  \left({{r-l}\choose b-m+j'(p-1)}\right)_{\substack{c-m\leq j'\leq c-1\\ m\leq l\leq b}} 
 \end{eqnarray*}
that are obtained from $A''$ and $B''$ respectively by putting $j' = c-b-1+j$. Now by Lemma \ref{coeff 5.1}, we have
$${{r-l}\choose b-m+j'(p-1)}\equiv {{p+b-c-l}\choose p+b-m-j'}{{c-1}\choose j'-1}\bmod p\quad\text{for}\quad c-m\leq j'\leq c-1$$
as $p+b-c\geq c-1$ and $b-m+1\leq c-m$.  Also,  note that
\begin{eqnarray}{\label{bino LT}}
 {{p+b-c-l}\choose p+b-m-j'} \equiv 0\bmod p \iff j'<c-m+l.
\end{eqnarray}
Therefore modulo $p$, $A''$ is invertible (being lower triangular with non-zero diagonal entries) and $B''$ is zero.  \\
\textbf{Subcase (ii)} $B'$ is invertible:\\
Lemma \ref{coeff 5.1} gives us $$B'\equiv \left({{p+b-c-l}\choose b-m-j}{{c-1}\choose j}\right)_{\substack{0\leq j\leq b-m\\m\leq l\leq b}}\bmod p.$$
Hence $B'$ is invertible mod $p$ iff $$\left({{p-c+l'-1}\choose p-c+l'-j'}\right)_{1\leq j',l'\leq b-m+1}$$is invertible mod $p$. The latter obtained by putting $j' = b-m-j+1$ and $l' = b-l+1$ and using ${a\choose b} = {a\choose a-b}$.  By Lemma \ref{grinberg} the determinant of this matrix is $1$, and thus $B'$ is invertible. \\

 From above it follows that $A$ is invertible.  Now for a fixed $0\leq j''\leq b$ let $\textbf{d}_{j''} =(d_0,d_1,...,d_{b})\in \mathbb{Z}^{b+1}_p$ be a vector such that $\textbf{d}_{j''} = A^{-1}\textbf{e}_{j''}$. Hence we have following system of equations
\begin{eqnarray}{\label{mono 1.2 eqn 1}}
\underset{0\leq l\leq b}\sum d_l  {{r-l}\choose b-m+j(p-1)} = \begin{cases}
1 \quad\text{if}\quad j = j'', \ 0\leq j\leq b-m\\
0\quad\text{if}\quad j\not= j'', \ 0\leq j\leq b-m
\end{cases}
\end{eqnarray}

\begin{eqnarray}{\label{mono 1.2 eqn 3}}
\underset{0\leq l\leq b}\sum d_l   {{r-l}\choose b-m+j'(p-1)}= \begin{cases}
1 \quad\text{if}\quad j' = j''+c-1-b, \ c-m\leq j'\leq c-1\\
0\quad\text{if}\quad j'\not= j''+c-1-b, \ c-m\leq j'\leq c-1.\\
\end{cases}
\end{eqnarray}
by putting $j' = c-1-b+j$ in (\ref{defn A case 1}).
Now we observe that Proposition \ref{gen 1} gives \\
$(T-a_p)\left(\sum_{0\leq l\leq b}d_lf^l\right) = \left[g, \ \underset{0\leq j\leq b-m}\sum\ \ \underset{0\leq l\leq b}\sum d_l   {{r-l}\choose b-m+j(p-1)}x^{r-b+m-j(p-1)}y^{b-m+j(p-1)}\right]\vspace*{2mm}$\\
$\hspace*{11em}+ \left[g, \ \underset{c-m\leq j'\leq c-1}\sum\ \ \underset{0\leq l\leq b}\sum d_l   {{r-l}\choose b-m+j'(p-1)}x^{r-b+m-j'(p-1)}y^{b-m+j'(p-1)}\right]$\vspace*{2mm}\\
where $f^l$ are as in Proposition \ref{gen 1},  observing that the sum for $b-m+1\leq j\leq c-m-1$ vanishes mod $p$ (since ${{r-l}\choose b-m+j(p-1)}\equiv 0\bmod p$ by Lemma \ref{coeff 5.1} together with $j<c-m$). 
Further, using (\ref{mono 1.2 eqn 1}) and (\ref{mono 1.2 eqn 3}) we have
$$(T-a_p)\left(\sum_{0\leq l\leq b}d_lf^l\right)\equiv [g,\ x^{r-b+m-i(p-1)}y^{b-m+i(p-1)}]$$
for $0\leq i\leq b-m$ or $c-m\leq i\leq c-1$.

\textbf{Case (ii)} $2c-1-p\leq b\leq c-1$ and $b\leq m \leq c-1$ \\
In this case we consider the matrix $A = (a_{j,l})$ over $\mathbb{Z}_p$ where
\[a_{j,l} = 
\begin{cases}
{{r-l}\choose b-m+j(p-1)} &\text{if}\quad 1\leq j\leq c-1, \ 0\leq l\leq  m-1 \vspace{2mm}\\
\frac{{{r-l}\choose b-m+j(p-1)}}{p} &\text{if}\quad 1\leq j\leq c-1, \ m\leq l\leq c-2.
\end{cases}
\]
Since $b-m+1\leq 1\leq j\leq c-1\leq p+b-m$ and $b\leq m\leq p+b-c$, then by Lemma \ref{coeff 5.1} and Lemma \ref{lmk68} we have 
\[a_{j,l}\equiv 
\begin{cases}
{{p+b-c-l}\choose p+b-m-j}{{c-1}\choose j-1} &\text{if}\quad 1\leq j\leq c-1, \ 0\leq l\leq m-1\vspace{2mm}\\
\frac{(-1)^{l-m}{{p+b-m-1}\choose j-1}{{p-1+m-l}\choose c-1-j}}{{{p+b-c-m}\choose l-m}{{p+b-m-1}\choose c-1}} &\text{if}\quad 1\leq j\leq c-1, \ m\leq l\leq c-2.
\end{cases}
\]
If $m\leq c-2$ then we can write $A$ as follows 
\begin{eqnarray}{\label{Block 2 A}}
A \equiv \begin{pmatrix}
 A'& B'\\
 A''&B''
 \end{pmatrix}\bmod p 
\end{eqnarray}
 where the ranges of $l$ and $j$ are divided into non-empty intervals $[0,\ m-1],[m, \ c-2]$ and $[1,c-m-1], [c-m,  c-1]$ respectively,  determining the order of the blocks.  If $m = c-1$ then we observe that $A = A''$ as $c-m = 1$ and $m-1 = c-2$.  Following the same argument given in Case (i) of Proposition \ref{mono 1} we see that $A'', \ B'$ are invertible mod $p$ and $A'$ is zero mod $p$. Thus, $A\in\mathrm{GL}_{c-1}\left(\mathbb{Z}_p\right)$ in both the cases.
Now for a fixed $1\leq j'\leq c-1$ let $\textbf{d}_{j'} =(d_0,d_1,...,d_{c-2})\in \mathbb{Z}^{c-1}_p$ be a vector such that $\textbf{d}_{j'} = A^{-1}\textbf{e}_{j'}$.  Taking $f = \sum_{0\leq l\leq c-2}d_l\frac{f^l}{p^\sigma}$ in Proposition \ref{gen 1} we get the required result, where $ \sigma$ is $1$ if $m\leq l\leq c-2$ and $0$ otherwise.

\textbf{Case (iii)} $2\leq b\leq 2c-2-p$ and $1 \leq m \leq b-1$\\
In this case we consider the following matrix $$A = (a_{j,l}) = \left({{r-l}\choose b-m+j(p-1)}\right)_{\substack{0\leq j\leq c-1\\ 0\leq l\leq c-1}}.$$
By Lemma \ref{coeff 5.1},  we have
\[ {{r-l}\choose b-m+j(p-1)}\equiv  
\begin{cases}
{{p+b-c-l}\choose b-m-j}{{c-1}\choose j} & \text{if}\quad 0\leq j\leq b-m, \ 0\leq l\leq p+b-c\vspace{1 mm}\\
{{2p+b-c-l}\choose b-m-j}{{c-2}\choose j} & \text{if}\quad 0\leq j\leq b-m, \ p+b-c+1\leq l\leq c-1\vspace{1 mm}\\
{{p+b-c-l}\choose p+b-m-j}{{c-1}\choose j-1} & \text{if}\quad b-m+1\leq j\leq c-1, \ 0\leq l\leq p+b-c\vspace{1 mm}\\
{{2p+b-c-l}\choose p+b-m-j}{{c-2}\choose j-1} & \text{if}\quad  b-m+1\leq j\leq c-1, \ p+b-c+1\leq l\leq c-1.\\
\end{cases}
\]
Now we write $A$ in blocks as follows $$A = \begin{pmatrix}
 A'& B'& C'&D'\\
 A''&B''&C''&D''\\
 A'''&B'''&C'''&D'''
\end{pmatrix}$$
where the ranges of $l$ and $j$ are divided into non-empty intervals $[0, \ m-1], \ [m, \ p-c+m-1], \ [p-c+m, \ p+b-c], \ [p+b-c+1, \ c-1]$ and $[0,\ b-m], \ [b-m+1, \ c-m-1], \ [c-m, \ c-1]$ respectively,  determining the order of the blocks. We refer to the argument using (\ref{bino LT}) in Case (i) to deduce that modulo $p$: (i) $A'''$ and $C'$ are invertible lower triangular and (ii) $A''$, $B''$, $C''$,  $B'''$, and $C'''$ are all zero.  Therefore we have 
 $$A\equiv\begin{pmatrix}
 A'& B'& C'&D'\\
 \textbf{0}& \textbf{0}& \textbf{0}&D''\\
 A'''& \textbf{0}& \textbf{0}&D'''
\end{pmatrix}\bmod p.$$
Now we observe that for $0\leq j\leq b-m$ or $c-m\leq j\leq c-1$, the $j^{th}$ row can not be written as a linear combination of rest of the rows because $C'$ and $A'''$ are invertible mod $p$.  So for a fixed $0\leq j'\leq b-m$ or $c-m\leq j'\leq c-1$, we claim there is a vector $\textbf{d}_{j'} = (d_0,\ d_1,\ ...,d_{c-1})\in\mathbb{Z}_p$ such that $A\cdot\textbf{d}_{j'} = \textbf{e}_{j'} \bmod p$. This is because modulo $p$, the row rank of the augmented matrix $[A|\textbf{e}_{j'}]$ is equal to the row rank of $A$. As before we invoke Proposition \ref{gen 1} to prove our claim. 

\textbf{Case (iv)} $2\leq b\leq 2c-2-p$ and $ b \leq m \leq p+b-c +1$\\
Here we consider the matrix $A = (a_{j,l})$ over $\mathbb{Z}_p$ where
\[a_{j,l} = 
\begin{cases}
{{r-l}\choose b-m+j(p-1)} &\text{if}\quad 1\leq j\leq c-1, \ 0\leq l\leq m-1, \ p+b-c+1\leq l\leq c-2\vspace{2mm}\\
\frac{{{r-l}\choose b-m+j(p-1)}}{p} &\text{if}\quad 1\leq j\leq c-1, \ m\leq l\leq p+b-c.

\end{cases}
\]
By Lemma \ref{coeff 5.1} and Lemma \ref{lmk68}
\[a_{j,l} \equiv 
\begin{cases}
{{p+b-c-l}\choose p+b-m-j}{{c-1}\choose j-1} &\text{if}\quad 1\leq j\leq c-1, \ 0\leq l\leq m-1\vspace{2mm}\\
\frac{(-1)^{l-m}{{p+b-m-1}\choose j-1}{{p-1+m-l}\choose c-1-j}}{{{p+b-c-m}\choose l-m}{{p+b-m-1}\choose c-1}} &\text{if}\quad 1\leq j\leq c-1, \ m\leq l\leq p+b-c\vspace{1mm}\\
{{2p+b-c-l}\choose p+b-m-j}{{c-2}\choose j-1} &\text{if}\quad 1\leq j\leq c-1, \ p+b-c+1\leq l\leq c-2.
\end{cases}
\]
Here we note that 
$${{2p+b-c-l}\choose p+b-m-j}{{c-2}\choose j-1} = {{p+b-m-1}\choose j-1}{{p-1+m-l}\choose c-1-j}\frac{(2p+b-c-l)!(c-2)!}{(p-1+m-l)!(p+b-m-1)!}.$$
 \[\implies a_{j,l} \equiv 
\begin{cases}
{{p+b-c-l}\choose p+b-m-j}{{c-1}\choose j-1} &\text{if}\quad 1\leq j\leq c-1, \ 0\leq l\leq m-1\vspace{1mm}\\
\beta_l{{p+b-m-1}\choose j-1}{{p-1+m-l}\choose c-1-j} &\text{if}\quad 1\leq j\leq c-1, \ m\leq l\leq c-2
\end{cases}
\]
where
\[\beta_l = 
\begin{cases}
\frac{(-1)^{l-m}}{{{p+b-c-m}\choose l-m}{{p+b-m-1}\choose c-1}} & \text{if}\quad m\leq l\leq p+b-c\vspace{1mm}\\
\frac{(2p+b-c-l)!(c-2)!}{(p-1+m-l)!(p+b-m-1)!} &\text{if}\quad p+b-c+1\leq l\leq c-2.
\end{cases}
\]
\vspace{2mm}

Now,  proceeding as in Case (ii) of Proposition \ref{mono 1} one shows that $A$ is invertible mod $p$.  So for a fixed $1\leq j''\leq c-1$ let $\textbf{d}_{j''} =(d_0,d_1,...,d_{c-2})\in \mathbb{Z}^{c-1}_p$ be a vector such that $\textbf{d}_{j''} = A^{-1}\textbf{e}_{j''}$. Taking $f = \sum_{0\leq l\leq c-2}d_l\frac{f^l}{p^\sigma}$ in Proposition \ref{gen 1} we get the required result, where $ \sigma$ is $1$ if $m\leq l\leq p+b-c$ and $0$ otherwise.
 
 \textbf{Case (v)} $2\leq b\leq 2c-3-p$  and $p+b-c+2\leq m\leq c-1$ \\
 We note that Case (iv) exhausts all the values of $m$ if $p+b-c = c-2$. Consider the following matrix 
 $$A = (a_{j,l}) = \left({{r-l}\choose b-m+j(p-1)}\right)_{\substack{1\leq j\leq c-1\\ 0\leq l\leq c-2}}.$$
By Lemma \ref{coeff 5.1},  we have
\[ {{r-l}\choose b-m+j(p-1)} = 
\begin{cases}
{{p+b-c-l}\choose p+b-m-j}{{c-1}\choose j-1} & \text{if}\quad 1\leq j\leq p+b-m, \ 0\leq l\leq p+b-c\vspace{1mm}\\
{{2p+b-c-l}\choose p+b-m-j}{{c-2}\choose j-1} & \text{if}\quad 1\leq j\leq p+b-m, \ p+b-c+1\leq l\leq c-2\vspace{1mm}\\
{{p+b-c-l}\choose 2p+b-m-j}{{c-1}\choose j-2} & \text{if}\quad p+b-m+1\leq j\leq c-1, \ 0\leq l\leq p+b-c\vspace{1mm}\\
{{2p+b-c-l}\choose 2p+b-m-j}{{c-2}\choose j-2} & \text{if}\quad  p+b-m+1\leq j\leq c-1, \ p+b-c+1\leq l\leq c-2.
\end{cases}
\]
If $m\leq c-2$ then $A$ can be written as follows
$$A = \begin{pmatrix}
 A'& B'& C'\\
 A''&B''&C''\\
 A'''&B'''&C'''
\end{pmatrix}$$
where the ranges of $l$ and $j$ are divided into non-empty intervals $[0, \ p+b-c], \ [p+b-c+1, \ m-1], \ [m, \ c-2]$ and $[1,\ c-m-1], \ [c-m, \ p+b-m], \ [p+b-m+1, \ c-1]$ respectively,  determining the order of the blocks. For $m = c-1$, $A$ is given by only the blocks $A''$, $B''$, $A'''$ and $B'''$ above.  By a similar argument in Case (i) above using (\ref{bino LT}),  one shows that modulo $p$: (i) $A''$, $B'''$ are invertible lower triangular, (ii) $A'$, $A'''$ and $C'''$ are zero. 
 
Next, observe that $C'$ is invertible mod $p$ if same holds for the matrix
$$C_1' = \left({{2(p-c)+b+2+l'-1}\choose p-c+1+l'-j'}\right)_{1\leq j', l'\leq c-m-1}.$$
The latter is obtained by putting $j' = c-m-j,\ l' = c-1-l$ and using the identity ${M\choose N} = {M\choose M-N}$.  By Lemma \ref{grinberg},  we deduce that $C_{1}'$ is invertible mod $p$.  Hence if $m\leq c-2$ then 
$$A\equiv  \begin{pmatrix}
 \textbf{0}& B'& C'\\
 A''&B''&C''\\
 \textbf{0}&B'''&\textbf{0}
\end{pmatrix}$$
where $C', \ A''$ and $B'''$ are invertible mod $p$. This gives in both cases including $m = c-1$ that $A$ is invertible mod $p$ (as $A \bmod p$ is of full row rank). Finally,  by the usual arguments using Proposition \ref{gen 1} we obtain the required result. 
\end{proof}

\begin{prop}{\label{gen 2}}
Let $r = s+p^t(p-1)d $, $\  s=b+c(p-1)$ such that $p\nmid d$, $2\leq b\leq p$ and $0\leq c\leq p-2$.  Fix $a_p$ such that $s>2\nu(a_p)$ and $c<\nu(a_p)<\text{min}\{\frac{p}{2}+c-\epsilon,\ p-1\}$ where $\epsilon$ is defined as in (\ref{dfn epsilon}).  Further assume that $t\geq 2\nu(a_p)$ if $b\geq 2c-1$ and $t>2\nu(a_p)+\epsilon-1$ if $b\leq 2c-2$.  Let $m$ be such that $1 \leq c+1-\epsilon \leq m\leq\lfloor{\nu(a_p)}\rfloor$  and $(b,c,m) \not= (p,0,1)$.  \\
(i) If $(b, m)\not= (2c-p+1, \ c)$ then for $0\leq l<m-\nu({{r-l}\choose{r-m}})$ there exists $f^l\in \indkg(\symqp)$ such that 
$$(T-a_p)(f^l)\equiv \frac{p^m}{a_p}\left[g^0_{1,0},\underset{\substack{c<j<s-m\\j\equiv (r-m)\bmod (p-1)}}\sum\frac{{{r-l}\choose j}}{{{r-l}\choose {r-m}}} x^{r-j}y^j \right] + \left[ g^0_{2,0},F_m(x,y)\right].$$
(ii)  If $(b, m) = (2c-p+1, \ c)$ then for $0\leq l<m-\nu({{r-l}\choose{r-m}})$ there exists $f^l\in \indkg(\symqp)$ such that 
$$(T-a_p)(f^l)\equiv \frac{p^m}{a_p}\left[g^0_{1,0},\underset{\substack{0\leq j<s-m\\j\equiv (r-m)\bmod (p-1)}}\sum\frac{{{r-l}\choose j}}{{{r-l}\choose {r-m}}} x^{r-j}y^j \right] + \left[ g^0_{2,0},F_m(x,y)\right].$$
\end{prop}
\begin{rk}{\label{rk 4.5}}
 The set $\left[0, \ m-\nu\left({{r-l}\choose r-m}\right)\right)\not =\phi$ as long as $(b,c,m) \not= (p, 0, 1)$ and $m\geq c+1-\epsilon$.  Hence in the above proposition, $l = 0$ always satisfies the condition $0\leq l<m-\nu\left({{r-l}\choose r-m}\right)$.
\end{rk}
\begin{proof}
We consider the following functions 
\begin{eqnarray*}
 f_3 &=&  \sum_{\lambda\in I^*_1} f_{3, \lambda}  = \sum_{\lambda\in I^*_1}\left[ g^0_{2,p\lambda}, \left(\frac{p}{\lambda}\right)^{m-l}\frac{F_l(x,y)}{(p-1){{r-l}\choose r-m}a_p}\right]\\
f_2 &=& \left[ g^0_{2,0}, \frac{-F_m(x,y)}{a_p}\right]\\
f_1 &=& \left[ g^0_{1,0}, \frac{p^m}{a^2_p}\underset{\substack{s-m\leq j<r-m\\ j\equiv (r-m)\bmod (p-1)}}\sum\frac{{{r-l}\choose j}}{{{r-l}\choose r-m}}x^{r-j}y^j\right]\\
f_0 &=& \begin{cases}
                 \left[ 1, \ \frac{p^{2m-b}{{r-l}\choose b-m}}{a_p{{r-l}\choose r-m}}F_{s-b+m}(x,y)\right] &\text{if}\ 0\leq b-m\leq c <b-m+p-1\vspace{2mm}\\
                 \left[ 1,\ \frac{p^{2m-b-(p-1){{r-l}\choose b-m+p-1}}}{a_p{{r-l}\choose r-m}}F_{s-(b-m+p-1)}(x,y)\right] &\text{if}\ b-m+p-1\leq c,\ (b,\ m) \not= (2c-p+1, \ c)\vspace{2mm}\\
                 0 & \text{otherwise.}
                \end{cases}
\end{eqnarray*}

First we note that $\nu\left({{r-l}\choose r-m}\right)\leq 1$ by Lemma \ref{rk3.15} which is used throughout the proof.  Now we compute $T^+$ and $T^-$ of the functions above.
\begin{eqnarray*}
T^+(f_2) &=& -\sum_{\lambda\in I^*_1}\left[ g^0_{3,p^2\lambda} \sum_{0\leq j\leq s-m}\frac{p^j(-\lambda)^{r-m-j}}{a_p}\left( {{r-m}\choose j}-{{s-m}\choose j}\right)x^{r-j}y^j\right]\hspace{12em}\\
&&\hspace*{6em}- \sum_{\lambda\in I_1}\left[ g^0_{3,p^2\lambda}\sum_{s-m+1\leq j\leq r-m}\frac{p^j{{r-m}\choose j}(-\lambda)^{r-m-j}}{a_p}x^{r-j}y^j\right]\hspace{12em}\\
&&\hspace*{19em}+\left[ g^0_{3,0}, \frac{p^{s-m}}{a_p}x^{r-s+m}y^{s-m}\right].
\end{eqnarray*}

We estimate the valuation of the coefficients in the above three sums labelled (I), (II) $\&$ (III).  For (I), $j+t-\nu(j!)-\nu(a_p)\geq t-\nu(a_p)>0$. For (III), $s-m-\nu(a_p)\geq s-2\nu(a_p)>0$. For (II), $j-\nu(a_p)\geq s-m+1-\nu(a_p)>0$,  giving us that $T^+(f_2)\equiv\ 0\bmod p$.  Now observe that for $T^{+} (f_{3, \lambda})$ we obtain three similar sums as above.  Therefore,  using the calculations above together with the assumption that $l<m-\nu\left({{r-l}\choose r-m}\right)$ allows us to see that the first two sums in $T^{+} (f_{3, \lambda})$ are also zero mod $p$.  Moreover, the last sum is also zero since $s-l-\nu(a_p)+m-l-\nu\left({{r-l}\choose r-m}\right)>s- m - \nu(a_p)>0$. This gives that $T^+(f_3)\equiv\ 0\bmod p$.

\begin{eqnarray*}
T^{-}(f_{3, \lambda})
 &=& \left[ g^0_{1,0}, \sum_{0\leq j\leq r-l}\frac{p^m{{r-l}\choose j}}{a_p(p-1){{r-l}\choose r-m}}\lambda^{r-m-j}x^{r-j}y^j\right]\\
&&- \left[ g^0_{1,0}, \sum_{0\leq j\leq s-l}\frac{p^{r-s+m}{{s-l}\choose j}}{a_p(p-1){{r-l}\choose r-m}}\lambda^{s-m-j}x^{r-j}y^j\right]\\
\implies\ \ T^-(f_3)&\equiv & \left[ g^0_{1,0}, \underset{\substack{0\leq j\leq r-l\\ j\equiv (r-m)\bmod (p-1)}}\sum\frac{p^m{{r-l}\choose j}}{a_p{{r-l}\choose r-m}}x^{r-j}y^j\right]
\end{eqnarray*}
 as $r-s+m-\nu(a_p)-\nu\left({{r-l}\choose r-m}\right)>0$.  Also,
\begin{eqnarray*}
T^-(f_2) &=& -\left[ g^0_{1,0},\frac{p^m}{a_p}x^my^{r-m}\right] +\left[ g^0_{1,0}, \frac{p^{r-s+m}}{a_p}x^{r-s+m}y^{s-m}\right]\vspace{3mm}\\
&\equiv & -\left[ g^0_{1,0},\frac{p^m}{a_p}x^my^{r-m}\right]\quad\text{as}\ r-s+m-\nu(a_p)>0.
\end{eqnarray*}
Now,
\begin{align*}
\hspace{3em} T^+(f_1)\  &=& \sum_{\lambda\in I^*_1}\left[g^0_{2,p\lambda},\sum_{0\leq j<r-m}\frac{p^{j+m}(-\lambda)^{r-m-j}}{a^2_p{{r-l}\choose r-m}}\underset{\substack{s-m\leq i<r-m\\i\equiv (r-m)\bmod (p-1)}}\sum{{r-l}\choose i}{i\choose j}x^{r-j}y^j\right]\\
&&+\left[g^0_{2,0}, \underset{\substack{s-m\leq j<r-m\\j\equiv (r-m)\bmod (p-1)}}\sum\frac{p^{j+m}}{a^2_p{{r-l}\choose r-m}}{{r-l}\choose j}x^{r-j}y^j\right].
\end{align*}
\\
Now we estimate the valuation of the coefficients in the above two sums (I) $\&$ (II).  For (II), when $j= s-m,$ using Lemma \ref{rk3.15} we note that the valuation is $s-2\nu(a_p)>0$.  When $j\geq s-m+1$ the valuation is at least $j+m-2\nu(a_p)-\nu\left({{r-l}\choose r-m}\right)\geq s-2\nu(a_p)+1-\nu\left({{r-l}\choose r-m}\right)>0$ as $\nu\left({{r-l}\choose r-m}\right)\leq 1$.  For (I) observe that the first summation truncates to $j\leq s-m$ by above calculation. Therefore for $0\leq j\leq s-m\leq s-l$, using (\ref{dfnsrjm}) we have 
$$T^+(f_1) \equiv \sum_{\lambda\in I^*_1}\left[g^0_{2,p\lambda},\sum_{0\leq j\leq s-m}\frac{p^{j+m}(-\lambda)^{r-m-j}}{a^2_p{{r-l}\choose r-m}}S_{r,j,l,m}\ x^{r-j}y^j\right]. $$
\\
For $c= 0$,  Lemma \ref{srjm} gives $\nu(S_{r,j,l,m})\geq t$ and so
$$j+m+t-2\nu(a_p)-\nu\left({{r-l}\choose r-m}\right)\geq m-\nu\left({{r-l}\choose r-m}\right) +t-2\nu(a_p)>0$$
since $m-{{r-l}\choose r-m}>0$ and $t\geq 2\nu(a_p)$.  Now for $c\geq 1$,  $p^{t-c+1}|S_{r,j,l,m}$ by Lemma \ref{srjm} ($c+m\geq 2$ holds in this case). Therefore, the valuation of the coefficients is at least
$$j+m+t-c+1-2\nu(a_p)-\nu\left({{r-l}\choose r-m}\right)\geq 1-\nu\left({{r-l}\choose r-m}\right) +t-2\nu(a_p)+m-c>0.$$
Observe that for the last inequality we also use the following: (i) if $m>c$ then $t\geq 2\nu(a_p)$,  and (ii) if $c+1 - \epsilon \leq m\leq c$ then $t>2\nu(a_p)+ \epsilon -1$. Hence $T^+(f_1)\equiv\ 0\bmod p$.  Next we have 
$$T^-{f_1} = \left[ 1,\underset{\substack{s-m\leq j<r-m\\ j\equiv (r-m)\bmod (p-1)}}\sum\frac{p^{r-j+m}{{r-l}\choose j}}{a^2_p{{r-l}\choose r-m}}x^{r-j}y^j\right].$$\\
Here we observe that valuation of coefficients above is at least 
$r-j+m-\nu\left({{r-l}\choose r-m}\right)-2\nu(a_p)\geq (p-1)+2m-1-2\nu(a_p)>p-1+2m-1-(p+2c-2\epsilon)\geq 2(m-c+\epsilon)-2>0.$
Here the second inequality follow from $\nu(a_p)<\frac{p}{2}+c-\epsilon$ and the last inequality follows from $m\geq c+1-\epsilon$.
Hence $T^+(f_1)$ and $T^-(f_1)$ are both congruent to zero mod$\ p$. Now we will compute $T^+(f_0),\ T^-(f_0)$ and $a_pf_0$ in respective cases. 

\textbf{Case (i)} $0\leq b-m\leq c<b-m+p-1$\\
Here we note that $m\geq c$ because $m = c-1$ gives $c<b-m+p-1 = b-(c-1)+p-1\implies b>2c-p$, but for this range of $b$ we have by  assumption that $m\geq c$. 
$$T^+(f_0) = \sum_{\lambda\in I^*_1}\left[ g^0_{1,\lambda}, \sum_{0\leq j\leq b-m}\frac{p^{j+2m-b}{{r-l}\choose b-m}(-\lambda)^{b-m-j}}{a_p{{r-l}\choose r-m}}\left({{r-s+b-m}\choose j}-{{b-m}\choose j}\right)x^{r-j}y^j\right]$$
$$\hspace{8em}+\sum_{\lambda\in I_1}\left[ g^0_{1,\lambda}, \sum_{b-m+1\leq j\leq r-s+b-m}\frac{p^{j+2m-b}{{r-l}\choose b-m}{{r-s+b-m}\choose j}(-\lambda)^{r-s+b-m-j}}{a_p{{r-l}\choose r-m}}x^{r-j}y^j\right]$$
$$\hspace{29em}-\left[g^0_{1,0}, \frac{p^{m}{{r-l}\choose b-m}}{a_p{{r-l}\choose r-m}}x^{r-b+m}y^{b-m}\right]$$\\
Here we note that the valuation of the coefficients in the first sum (I) is at least
\begin{eqnarray*}
j+2m-b+t-\nu(j!)-\nu(a_p)-\nu\left({{r-l}\choose r-m}\right)&\geq & m-b+t-\nu(a_p)+m-\nu\left({{r-l}\choose r-m}\right)\\
&>& m-b+\nu(a_p)\geq\ \nu(a_p)-c>0.
\end{eqnarray*}
We deduce that the sum in (II) is also zero mod $p$ using the above inequalities and the fact that $\nu\left({{r-s+b-m}\choose j}\right)\geq t-\nu(j!)$ for $j\geq b-m+1$. 
Therefore we have 
 $ T^+( f_0)\equiv \left[ g^0_{1,0}, \frac{- p^{m}{{r-l}\choose b-m}}{a_p{{r-l}\choose r-m}}x^{r-b+m}y^{b-m}\right]$. Further,
\begin{eqnarray*}
\hspace{6em} T^-(f_0) &=& \left[\alpha, \frac{p^{s+3m-2b}{{r-l}\choose b-m}}{a_p{{r-l}\choose r-m}}x^{s-b+m}y^{r-s+b-m}-\frac{p^{r+3m-2b}{{r-l}\choose b-m}}{a_p{{r-l}\choose r-m}}x^{r-b+m}y^{b-m}\right].
\end{eqnarray*}
We use $0\leq b-m\leq c$ and Lemma \ref{rk3.15} to give the estimate below of the valuation of the coefficient of the first term: 
\begin{eqnarray*}
s+3m-2b-\nu(a_p)-\nu\left({{r-l}\choose r-m}\right)
 &= & \begin{cases}
 s+m-2(b-m)-\nu(a_p)-1 &\text{if}\ m\geq b-c+1,\ 0\leq l\leq b-c\\
 s+m-2(b-m)-\nu(a_p) &\text{else}
 \end{cases}\\
 &\geq & \begin{cases}
 s+m-2c-\nu(a_p)+1 &\text{if}\ m\geq b-c+1,\ 0\leq l\leq b-c\\
 s+m-2c-\nu(a_p) &\text{else}.
 \end{cases}
\end{eqnarray*}
Therefore $s+3m-2b-\nu(a_p)-\nu\left({{r-l}\choose r-m}\right)\geq  s+m-2c-\nu(a_p)>\nu(a_p)-c+m-c>0$ as $s>2\nu(a_p)$ and $\nu(a_p)>c$. The second term is also zero mod $p$ by the same calculation above and observing that $r > s$. Therefore $T^-(f_0)\equiv\ 0\bmod p$.  Next for $a_{p} f_{0}$, using Lemma \ref{rk3.15} (note that $0\leq b-c\leq m<b-c+p-1$) we have:  

$$\nu\left(\frac{p^{2m-b}{{r-l}\choose b-m}}{{{r-l}\choose r-m}}\right) \geq \begin{cases}
                                                                        2m-b-1\quad\text{if}\quad m\geq b-c+1,\ 0\leq l\leq m-c\\
                                                                        2m-b \quad\text{otherwise}
                                                                         \end{cases}$$
 $$\nu\left(\frac{p^{2m-b}{{r-l}\choose b-m}}{{{r-l}\choose r-m}}\right)\geq \begin{cases}
                                                                        m-(c-1)-1\quad\text{if}\quad m\geq b-c+1,\ 0\leq m-c\\
                                                                        m-c \quad\text{otherwise}
                                                                         \end{cases}$$
                                                                         
giving that $\nu\left(\frac{p^{2m-b}{{r-l}\choose b-m}}{{{r-l}\choose r-m}}\right)\geq m-c$ in all cases. If  $m \geq c+1$ then $\nu\left(\frac{p^{2m-b}{{r-l}\choose b-m}}{{{r-l}\choose r-m}}\right)\geq 1$.  If $m =c$ then $2m-b-1 = 2c-b-1 \geq 1$ since we also have $b \leq 2c-2$.  Hence $a_pf_0\equiv 0\bmod p$ in all cases.

\textbf{Case (ii)} $0\leq b-m+p-1\leq c$ and $(b, \ m)\not= (2c-p+1, \ 0)$\\
In this case we have $c\geq 3$ as $b\geq 2$ and $ m<p-1$.  Let $c_0:= \frac{p^{2m-b-(p-1){{r-l}\choose b-m+p-1}}}{a_p{{r-l}\choose r-m}}$.
$$T^+(f_0) = \underset{\lambda\in I_1}\sum\left[g^0_{1,\lambda}, \ \underset{0\leq j\leq b-m+p-1}\sum p^jc_0(-\lambda)^{b-m-j}\left({{r-s+b-m+p-1}\choose j}- {{b-m+p-1}\choose j}\right)x^{r-j}y^j\right]$$
$$\hspace{5em}+\underset{\lambda\in I_1}\sum\left[ g^0_{1, \lambda},\ \underset{b-m+p\leq j\leq r-s+b-m+p-1}\sum p^jc_0{{r-s+b-m+p-1}\choose j}(-\lambda)^{r-s+b-m+p-1-j}x^{r-j}y^j\right]$$
$$\hspace{26em}-\left[ g^0_{1,0}, \ p^{b-m+p-1}c_0x^{r-s+b-m+p-1}y^{b-m+p-1}\right].$$\\
Here we note that $\nu\left({{r-s+b-m+p-1}\choose j}-{{b-m+p-1}\choose j}\right)\geq t-\nu(j!)$ and $j\geq b-m+p$ gives $\nu\left({{r-s+b-m+p-1}\choose j}\right)\geq t-\nu(j!)$. Hence the valuation of the coefficients in the first two sums is at least $j+\nu(c_0)+t-\nu(j!)\geq t+\nu(c_0)>0$. The last inequality holds since
$t+\nu(c_0) = t-\nu(a_p)-(b-m+p-1)+m-\nu\left({{r-l}\choose r-m}\right)+\nu\left({{r-l}\choose b-m+(p-1)}\right)>\nu(a_p)-c>0.$ Therefore,
$$T^+(f_0)\equiv -\left[ g^0_{1,0}, \ \frac{p^m{{r-l}\choose b-m+p-1}}{a_p{{r-l}\choose r-m}} x^{r-s+b-m+p-1}y^{b-m+p-1}\right]\bmod p.$$
Now,
$$T^-(f_0) = \left[\alpha,\ p^{s-(b-m+p-1)}c_0x^{s-(b-m+p-1)}y^{r-s+b-m+p-1}- p^{r-(b-m+p-1)}c_0x^{r-(b-m+p-1)}y^{b-m+p-1}\right].$$
The valuation of the coefficients above is at least
\begin{eqnarray*}
 s-(b-m+p-1)+\nu(c_0) &\geq& m+(c-1)(p-1)+2m-b-(p-1)-\nu(a_p) -\nu\left({{r-l}\choose r-m}\right)\\
&= & (c-2)(p-1)+p-1-\nu(a_p)+m-(b-m+p-1)+m-\nu\left({{r-l}\choose r-m}\right)\\
&>& (c-2)(p-1) + m-c> 0\quad \text{as}\quad c\geq 3.
\end{eqnarray*}
Hence we have $T^-(f_0)\equiv 0\bmod p$. Now we will estimate the valuation of the coefficient of $a_pf_0$. By Lemma \ref{rk3.15} (B)
\begin{eqnarray*}
\nu\left(a_pc_0\right)
&\geq &\begin{cases}
              2m-b-(p-1)-1 &\text{if}\quad b-c+p+1\leq m\leq b-c+2p,\ 0\leq l\leq m-c+1\\
               2m-b-(p-1)&\text{otherwise}.

              \end{cases}\\
&\geq & \begin{cases}
              m-c+1 &\text{since }\quad b-c+p+1\leq m \\
               m-c &\text{since } b-m+p-1\leq c.

              \end{cases}
\end{eqnarray*}
Hence $\nu(a_pc_0)>0$ if $m>c$ and also if $m = c$ in the first case.  Further,  we observe that the second case occurs only if $b-c +p-1 \leq m $ giving us $ b + p-1 \leq 2c$ if $m=c$.  Thus in this case,  $\nu(a_{p} c_{0}) \geq 2c -b-(p-1) > 0$ if $m=c$ as long as $c\not = \frac{b+p-1}{2}$.  
Lastly $m = c-1$ occurs only if $b\leq 2(c-1)-(p+1)$ thus in this case  $\nu(a_pc_0)\geq 2m-b-(p-1)-1 = 2(c-1)-p-b\geq 1.$ 
Therefore we have $a_pf_0\equiv 0\bmod p$ in all cases. \\
Also note that as $m-\nu\left({{r-l}\choose r-m}\right)>l\geq 0$, we have: 
$$-a_pf_3 = \sum_{\lambda\in I^*_1}\left[ g^0_{2,p\lambda},\left( {\frac{p}{\lambda}}\right)^{m-l}\frac{F_l(x,y)}{(p-1){{r-l}\choose r-m}}\right]\equiv\ 0\bmod p.$$
Thus to summarize:
\begin{eqnarray*}
(T-a_p)(f_3)&\equiv &\ \left[ g^0_{1,0}, \underset{\substack{0\leq j\leq r-l\\j\equiv\ (r-m)\bmod (p-1)}}\sum\frac{p^m{{r-l}\choose j}}{a_p{{r-l}\choose r-m}}x^{r-j}y^j\right]\\
(T-a_p)(f_2)&\equiv &-\left[ g^0_{1,0},\frac{p^m}{a_p}x^my^{r-m}\right]+\left[ g^0_{2,0}, F_m(x,y)\right]\\
(T-a_p)(f_1)&\equiv &- \left[ g^0_{1,0}, \frac{p^m}{a_p}\underset{\substack{s-m\leq j<r-m\\j\equiv (r-m)\bmod (p-1)}}\sum\frac{{{r-l}\choose j}}{{{r-l}\choose r-m}}x^{r-j}y^j\right]\\
(T-a_p)( f_0)&\equiv&- \left[ g^0_{1,0}, \frac{p^m}{a_p} \underset{\substack{0\leq j\leq c\\ j\equiv (r-m)\bmod (p-1)}}\sum\frac{{{r-l}\choose j}}{{{r-l}\choose r-m}}x^{r-j}y^j\right]\quad\text{if}\quad (b,\ m)\not=(2c-p+1, \ c)
\end{eqnarray*}

\noindent and $(T-a_p)(f_0) = 0$ if $(b,\ m)=(2c-p+1, \ c)$. Hence $f= f_3+f_2+f_1+f_0$ is the required function. 
\end{proof}

\begin{prop}{\label{other generator}}
Let $r = s+p^t(p-1)d$,  $s = b+c(p-1)$ such that $p\nmid d$,  $2\leq b\leq p$ and $1\leq c\leq p-2$. Suppose $c<\nu(a_p)<p-1$ and $1\leq m\leq c-1-\epsilon$. If $t>\nu(a_p)+c$ then 
\begin{equation}{\label{fm congruency}}
x^{r-b+m-(c-m-a)(p-1)}y^{b-m+(c-m-a)(p-1)}\equiv (-1)^m{{m+a-1}\choose a-1} F_m(x,y)\bmod \left( V^{m+1}_r+\kerp\right)
\end{equation}
 for $1\leq a\leq c-m-\epsilon$ where $\epsilon$ is defined in (\ref{dfn epsilon}). Further if $2\leq b\leq 2(c-1)-(p+1)$ and $m = b-1$ then (\ref{fm congruency}) holds for $1\leq a\leq c-m-1$.
\end{prop}
\begin{proof}
We begin by observing that if $2\leq b\leq 2(c-1)-(p+1)$ then $\epsilon = 2$ by hypothesis,  and so  (\ref{fm congruency}) holds for $1\leq a\leq c-m-2$ but if we take $m = b-1$ then we will prove (\ref{fm congruency}) actually holds for $1\leq a\leq c-m-1$.  Secondly, by Remark 4.4 of \cite{KBG} $F_m(x,y)\equiv x^{r-s+m}y^{s-m}\bmod (\kerp)$ which we use  later.
 
 Now let us consider $P_j := x^{r-(b+1+(c-j+1)(p-1))}y^{b-2m-1+(c-m-j)(p-1)}$ for $1\leq m\leq c-1-\epsilon$ and $0\leq j\leq c-m-\epsilon$. We claim that $P_j$ is a monomial, that is, the exponents of $x$ and $y$ are all non negative. The exponent of $x$ is non negative since $r> b+1+(c-j+1)(p-1)$ as $t\geq 2$ and $d\geq 1$.  And the exponent of $y$ is
\begin{eqnarray*}
 b-2m-1+(c-m-j)(p-1)&\geq & b-2m-1+\epsilon(p-1)\\
&\geq & b-2(c-1-\epsilon)-1+\epsilon(p-1)\\
&= & b-2(c-1)-1+\epsilon(p+1)\geq 0.
\end{eqnarray*}
The last inequality is clear if $\epsilon = 2$. It also follows for $\epsilon = 0$ and $\epsilon = 1$ since we have the conditions $b\geq 2c-1$ and $b\geq 2(c-1)-p$ for the corresponding values of $\epsilon$.  We note that if $m = b-1$ (for $b\leq 2(c-1)-(p+1)$) then above $P_j$ is a well defined monomial for $0\leq j\leq c-m-1$. This is because the exponent of $y$ is at least $b-2m-1+(p-1) = p-1-m\geq 0$. Hence in both case we observe that $P_j\in V_{r-(m+1)(p-1)}$ as the sum of the exponent of $x$ and $y$ is $r-(m+1)(p-1)$. Therefore,
\begin{equation}{\label{theta multi}}
\Theta^{m+1}P_j =\sum_{0\leq i\leq m+1}(-1)^i{m+1\choose i}x^{r-b+m-(c-m-j+i)(p-1)}y^{b-m+(c-m-j+i)(p-1)}.
\end{equation}
Next, we prove by induction that 
\begin{equation}{\label{claim}}
x^{r-b+m-(c-m-a)(p-1)}y^{b-m+(c-m-a)(p-1)}\equiv (-1)^m\eta_a F_m(x,y)\bmod \left( V^{m+1}_r+\kerp\right)
\end{equation}
for $1\leq a\leq c-m-\epsilon$ and for $1\leq a\leq c-m-1$ if $m = b-1$ (in case of $\epsilon = 2$ ) where 
\[ \eta_a = \begin{cases}
                 1  \quad & \quad \text{for}\quad a=1\\
                  \sum_{1\leq i\leq a-1}(-1)^{i+1}{{m+1}\choose i}\eta_{a-i} \quad & \quad \text{for}\ 2\leq a\leq c-m.
                  \end{cases}
                  \] 
Now putting $j = 1$ in (\ref{theta multi}) gives 
$$\sum_{0\leq i\leq m+1}(-1)^i{m+1\choose i}x^{r-b+m-(c-m-1+i)(p-1)}y^{b-m+(c-m-1+i)(p-1)}\equiv 0\bmod \left( V^{m+1}_r+\kerp\right).$$
We observe that except the first and the last term, all the terms belong to the kernel of $P$ by Proposition \ref{mono 1} and Proposition \ref{mono 1.2} since $1\leq i\leq m$ implies $c-m\leq c-m-1+i\leq c-1$. Therefore we get\\
$$x^{r-b+m-(c-m-1)(p-1)}y^{b-m+(c-m-1)(p-1)} \equiv  (-1)^m x^{r-b+m-c(p-1)}y^{b-m+c(p-1)}\bmod \left( V^{m+1}_r+\kerp\right)$$
$$\hspace*{14em}\equiv  (-1)^m\ \eta_1\ F_m(x,y)\bmod \left( V^{m+1}_r+\kerp\right).$$
This proves (\ref{claim}) for $a = 1$. Now by induction, we assume (\ref{claim}) holds for $1\leq a\leq n-1$ and prove the same for $a = n$.  Again by putting $j = n$ in (\ref{theta multi}) (noting that $n \leq c-m-\epsilon$ in general and $n \leq c-m-1$ in case of $m = b-1$ and $\epsilon = 2$),  we get 
$$\sum_{0\leq i\leq m+1}(-1)^i{m+1\choose i}x^{r-b+m-(c-m-n+i)(p-1)}y^{b-m+(c-m-n+i)(p-1)}\equiv 0\bmod \left( V^{m+1}_r+\kerp\right).$$
Here we observe that if $2 \leq n\leq m+1 $ then by Proposition \ref{mono 1} and Proposition \ref{mono 1.2} the above sum over $n\leq i\leq m+1$ belongs to $\kerp$.  If $n \geq m+2$ then ${{m+1}\choose i} = 0$ for all $m+1< i\leq n-1$.  So in either case, we have 
$$\sum_{0\leq i\leq n-1}(-1)^i{m+1\choose i}x^{r-b+m-(c-m-n+i)(p-1)}y^{b-m+(c-m-n+i)(p-1)}\equiv 0\bmod \left( V^{m+1}_r+\kerp\right).$$
For $1\leq i\leq n-1$, by induction $$x^{r-b+m-(c-m-(n-i))(p-1)}y^{b-m+(c-m-(n-i))(p-1)}\equiv (-1)^m\  \eta_{n-i}\ F_m(x, y)\bmod \left( V^{m+1}_r+\kerp\right)$$
$$ \implies x^{r-b+m-(c-m-n)(p-1)}y^{b-m+(c-m-n)(p-1)}\equiv (-1)^m\  \eta_n\ F_m(x, y)\bmod \left( V^{m+1}_r+\kerp\right).$$
Now using induction on $a$ and Lemma \ref{cmbi 1}, we can prove $\eta_a = {{m+a-1}\choose a-1}$. This completes the proof of our proposition.
\end{proof}
\section{Elimination of JH Factors }
\begin{prop}{\label{m<slope}}
Let $r = s+p^t(p-1)d$, $s = b+c(p-1)$ such that  $p\nmid d$,  $2\leq b\leq p$ and $0\leq c\leq  p-2$. Suppose that $s\geq 2c$ and $c<\nu(a_p)<p-1$.  Further we also assume $t\geq 2\nu(a_p)$ then there is a surjection 
$$ \ind^G_{KZ}\left( \frac{V^{(c-\epsilon)}_r}{V^{(\lfloor{\nu(a_p)}\rfloor+1)}_r}\right)\rightarrow\bar{\Theta}_{r+2,a_p}$$
where $\epsilon$ is defined as in (\ref{dfn epsilon}) and the map is induced from $P: \ind^G_{KZ}V_r\rightarrow\bar{\Theta}_{r+2,a_p}$.
\end{prop}

\begin{proof}
By Remark $4.4$ in \cite{KBG}, we have $\ind^G_{KZ}V^{(n)}_r\subset \kerp$ if $r\geq n(p+1)$ and $n>\nu(a_p)$. Using this fact for $n = \lfloor{\nu(a_p)}\rfloor+1$, we have $\ind^G_{KZ}\left(V^{(\lfloor{\nu(a_p)}\rfloor+1)}_r\right)\subset \kerp$ for $r\geq (\lfloor{\nu(a_p)}\rfloor+1)(p+1)$. For $r<(\lfloor{\nu(a_p)}\rfloor+1)(p+1)$, note that $V^{(\lfloor{\nu(a_p)}\rfloor+1)}_r = 0$. Hence, the surjection $P$ factors through $\ind^G_{KZ}\left(\frac{V_r}{V^{(\lfloor{\nu(a_p)}\rfloor+1)}_r}\right)$. This proves the proposition in the case when $c = 0$ since here we have $\epsilon = 0$.
Henceforth we assume that $c\geq 1$.

\textbf{Case (i)} $m =0$\\
\underline{\textbf{Subcase (i)}} For $2\leq b\leq p-1$\\
If $b\leq c-1$ then by Remark $4.4$ of \cite{KBG} $ x^{r-b}y^b\in \kerp$ as $b\leq c-1<\nu(a_p)$. If $c\leq b\leq p-1$ then Proposition \ref{gen 1} with $l = 0$ gives $$ [ g, \underset{\substack{0<j<s \\ j\equiv r\bmod p-1}}\sum\frac{{r\choose j}}{p}x^{r-j}y^j]\in \kerp.$$
But $x^{r-j}y^j\equiv x^{r-\bar{j}}y^{\bar{j}}\bmod\left(V^{(1)}_r\right)$ where $\bar{j}\equiv j\bmod (p-1)$ and $2\leq \bar{j}\leq p$
$$\implies \ \ \ \ \ \ \underset{\substack{0<j<s \\ j\equiv r\bmod p-1}} \sum\frac{{r\choose j}}{p}x^{r-j}y^j\equiv \eta x^{r-b}y^b \bmod \left( V^{(1)}_r\right)$$
where 
\begin{eqnarray*}
\eta &=& \underset{\substack{0<j<s\\ j\equiv s\bmod(p-1)}}\sum\frac{{r\choose j}}{p} \equiv\underset{\substack{0<j<s\\ j\equiv s\bmod(p-1)}}\sum\frac{{s\choose j}}{p}\equiv  \frac{b-s}{b}\not\equiv  0\bmod p.
\end{eqnarray*}
Here the first congruency follows since $\frac{{r\choose j}}{p}\equiv \frac{{s\choose j}}{p}\bmod p^{t-\nu(j!)} \ \text{and} \ \nu(j!)\leq \nu(s-(p-1)!)\leq c-1$, and the second last congruency follows from Lemma 2.5 in \cite{BG}.
Using $(4.2)$ of \cite{glover} and Lemma $5.3$ of \cite{Br03b}, we can see that the monomial $x^{r-b}y^b$ generates the quotient $V_{p-1-b}\otimes D^b$ of $\frac{V_r}{V^{(1)}_r}$ and $x^r$ generates the submodule $V_b$ of $\frac{V_r}{V^{(1)}_r}$, and the latter belongs to $\kerp$ by \cite{KBG}. 
Now let $$q'_0 = \underset{\substack{0<j<s \\ j\equiv r\bmod p-1}}\sum\frac{{r\choose j}}{p}x^{r-j}y^j$$ and we define $W_0$ in this case as the submodule generated by $x^r$ and $q'_0$.  Observe that $W_0$ satisfies all the required conditions of Lemma \ref{lm m<c}. \\
\underline{\textbf{Subcase (ii)}} $b = p$ \\
In this case by using (\ref{r' = p}) we have the following
$$0\longrightarrow V_1\longrightarrow\frac{V_r}{V^{(1)}_r}\longrightarrow V_{p-2}\otimes D\longrightarrow 0.$$
In the above exact sequence,  the first map sends $x$ to $x^r$ and the second map sends $x^{r-1}y$ to $x^{p-2}$. By the Remark $4.4$ of \cite{KBG}, we have $x^r,\ x^{r-1}y\in\kerp$ as $1\leq c<\nu(a_p)$.  We define  $W_0$ in this case as the submodule generated by $x^r$ and $x^{r-1}y$, and observe that $W_0$ satisfies the required conditions of Lemma \ref{lm m<c}.

From here onwards we will assume $m\geq 1$ and organise the proof accordingly as\linebreak $m\in \{[1,\ b-1]\cup[b, \ c-1-\epsilon]\}\cap[1,\ c-1-\epsilon]$. 

\textbf{Case (ii)} $1\leq m\leq b-1$\\
In this case by Proposition \ref{mono 1}  and Proposition \ref{mono 1.2}, for $0\leq j\leq \text{min}\{ b-m, \ c-1\}$ the monomials $q_j := x^{r-b+m-j(p-1)}y^{b-m+j(p-1)}$ are in $\kerp$. Further $1\leq m\leq c-1-\epsilon$, and so by Proposition \ref{other generator} the monomials $q_j\equiv {{c-1-j}\choose m} F_m(x,y)\bmod \left(V^{(m+1)}_r+\kerp\right)$ for $\epsilon\leq j\leq c-m-1$ and for $1\leq j\leq c-m-1$ if $(\epsilon, m) = (2,b-1)$. Here we observe that $[0, \ b-m]\cap[\epsilon, \ c-m-1]\not= \Phi$ because it contains $j = \epsilon$ if $(\epsilon, \ m)\not = (2, \ b-1)$ and $j = \epsilon -1$ if $(\epsilon, \ m) = (2, \ b-1)$.

\textbf{Case (iii)} $b\leq m\leq c-1-\epsilon$\\
In this case,  by Proposition \ref{mono 1.2} the monomials $q_j = x^{r-b+m-j(p-1)}y^{b-m+j(p-1)}\in\kerp$ for $1\leq j\leq c-1$. Since $m\leq c-1-\epsilon$,  Proposition \ref{other generator} gives $q_j\equiv {{c-1-j}\choose m} F_m(x,y)\bmod \left(V^{(m+1)}_r+\kerp\right)$ for $\epsilon\leq j\leq c-m-1$.  Here we note that $j = \epsilon\in[1,\ c-1]\cap[\epsilon, \ c-m-1]$ since $\epsilon\geq 1$ as $b\leq c-1$.\\

Now we observe that ${{c-1-j}\choose m}\not\equiv 0\bmod p$ for $j\leq c-m-1$ and $m\leq c\leq p-1$. We also observe that $q_j = {{c-1-j}\choose m}F_m(x,y)+v_{m+1}+\alpha_m$ for some $v_{m+1}\in V^{(m+1)}_r$,  $\alpha_m\in\kerp$ and $q_j\in\kerp$, where $j = \epsilon$ if $(\epsilon, m) \not= (2,b-1)$ and $j = \epsilon-1$ if $(\epsilon, m) = (2,b-1)$.  For $1\leq m\leq c-1-\epsilon$ we define $W_m$  to be the submodule of $V_r$ generated by ${{c-1-j}\choose m}F_m(x,y)+v_{m+1}$. Now we note that $F_m(x,y)\in V^{m}_r$ generates $\ind^G_{KZ}\left(\frac{V^{(m)}_r}{V^{(m+1)}_r}\right)$ using Lemma \ref{lm3.2},  which is applicable since $s>2m$ (as $m\le c-1-\epsilon$ and $s\geq 2c$).  This gives $W_m\subset \left(V^{(m)}_r\cap\kerp\right)$ and it also surjects onto $\frac{V^{(m)}_r}{V^{(m+1)}_r}$.
Now we observe that taking $W_m$ as above in Lemma \ref{lm m<c} with $0\leq m\leq c-1-\epsilon$ gives our result. 
\end{proof}  

\begin{prop}{\label{m>c}}
Let $r = s+p^t(p-1)d $, $\  s=b+c(p-1)$ such that $p\nmid d$, $2\leq b\leq p$ and $0\leq c\leq p-2$. Fix $a_p$ such that $s>2\nu(a_p)$ and $c<\nu(a_p)<\text{min}\{\frac{p}{2}+c-\epsilon,\ p-1\}$ where $\epsilon$ is defined as in (\ref{dfn epsilon}).  Further assume that $t\geq 2\nu(a_p)$ if $b\geq 2c-1$ and $t>2\nu(a_p)+\epsilon-1$ if $b\leq 2c-2$.  Then:\\
(i) If $(b,c) \not= (p,0)$ then there is a surjection 
$$\ind^G_{KZ}\left( \frac{V_r}{V^{(c+1-\epsilon)}_r}\right)\rightarrow\bar{\Theta}_{k',a_p}.$$
(ii) For $(b,c) = (p,0)$ there is a surjection 
$$\ind^G_{KZ}\left( \frac{V_r}{V^{(2)}_r}\right)\rightarrow\bar{\Theta}_{k',a_p}.$$
\end{prop}
\begin{proof}
Since the result is known for $0 < \nu(a_p) < 1$,  we assume that $\nu(a_p)\geq1$,  and so $t \geq 2$ by hypotheses. We show below that $P([g, \ F_m(x, y)]) = 0$ for $c+1-\epsilon\leq m\leq \lfloor\nu(a_p)\rfloor$ if $(b, \ c) \not= (p, \ 0)$ and for $2\leq m\leq \lfloor\nu(a_p)\rfloor$ if $(b, \ c) = (p, \ 0)$. \\

If $c = 0$ then the sum in Proposition \ref{gen 2} is empty,  and so we have $(T-a_p)f^l = \left[g^0_{2,0},F_m(x,y)\right] $,  where $1\leq m\leq\lfloor{\nu(a_p)}\rfloor$ if $b\leq p-1$ and $2\leq m\leq\lfloor{\nu(a_p)}\rfloor$ if $b= p$.  
So now we assume $c\geq 1$ and organise the proof accordingly as $m$ lies in one of the intervals in \\
$\{[1, \ b-c)\cup[b-c,\ p-1+b-c)\cup[b-c+p-1,\ b-c+2(p-1))\}\cap  [c+(1-\epsilon), \lfloor{\nu(a_p)}\rfloor].$

\textbf{Case (i)} $1\leq m<b-c$ \\
Observe that $b-c> m\geq c+1-\epsilon\implies b>2c+1-\epsilon \geq 2c-1$, hence by hypothesis $m\geq c+1$.  In this case Lemma \ref{rk3.15} implies that $\nu({{r-l}\choose {r-m}})= 0$ for $l=0,1,...,m-1$. We consider  the following matrix $A = (a_{j,i})\in M_{c+1}(\mathbb{Z}_p)$ given by 
\[
    a_{j,i} =     
                     \begin{cases}
                      \frac{{{r-(m-1-i)}\choose{j(p-1)+b-m}}}{{{r-(m-1-i)}\choose{r-m}}} & \text{if}\ \ 0\leq j\leq c-1,\ 0\leq i\leq c\\
                       1                                                                                                      & \text{if} \ \  j = c,\ 0\leq i\leq c           \\
                         
                         \end{cases}
   \]        
   By using Lucas' Theorem we have 
   $$\det(A) \equiv\frac{\Pi_{0\leq j\leq c}{c\choose j}\cdot \det(B)}{ \Pi_{0\leq i\leq c}{{r-(m-1-i)}\choose r-m}}\bmod p$$
   where $B = (b_{j,i}) ,b_{j,i} = {{b-m-c+1+i}\choose b-m-j}$.  Now we note that $B$ is invertible mod $p$ by Lemma \ref{invmt 2}. Therefore $A\in \mathrm{GL}_{c+1}(\mathbb{Z}_p)$ since the above multiplicative factor is unit.   So taking column vector $ \textbf{d} = (d_0,d_1,...,d_c)^t = A^{-1}(0,0,...,0,1)^t\in\mathbb{Z}^{c+1}_p$ gives the following
  $$\sum_{0\leq i\leq c-1}d_i\frac{{{r-(m-1-i)}\choose j(p-1)+b-m}}{{{r-(m-1-i)}\choose r-m}} = 0 \quad\text{for} \quad \ 0\leq j\leq c-1$$
$$\hspace{2em}\sum_{0\leq i\leq c} d_i = 1 \quad\text{for}\quad j = c.$$
 
Observe that Proposition \ref{gen 2}(i) is applicable for $0\leq l\leq m-1$ because by Lemma \ref{rk3.15},  ${{r-l}\choose r-m} = 0\quad \forall\quad 0\leq l\leq m-1$.  Therefore,  we can take $f = \sum_{0\leq i \leq c}d_if^{m-1-i}$, where $f^{m-1-i}$ are as in Proposition \ref{gen 2}(i) ($0\leq m-1-c\leq m-1-i\leq m-1$).  Hence $(T-a_p)(f) \equiv\left[ g^0_{2,0}, F_m(x,y)\right]$ for $c+1\leq m<b-c$. 

\textbf{Case (ii)} $b-c\leq m< (p-1)+b-c$ \\
We begin by observing that $m = c-1 $ is not possible in this case since the above constraint with $m = c-1$ gives $2c < b+p$, whereas we must have $2c \geq b+p+3$ if $m=c-1$.
For $c = 1$,  using Remark \ref{rk 4.5} and the inequalities above, we take $l = 0$ in Proposition \ref{gen 2} to get $(T-a_p)(f^0) \equiv\left[ g^0_{2,0}, F_m(x,y)\right]$ for the above values of $m$.  For $c\geq 2$, we consider  the following matrix $A = (a_{j,i})\in M_c(\mathbb{Z}_p)$ where
\begin{eqnarray*}
a_{j,i} &=&   
                    \begin{cases}
                    \frac{{{r-(b-m+j(p-1))}\choose i}}{{m\choose i}} & \text{if}\ \ 1\leq j\leq c-1, 0\leq i\leq c-1\\
                     1                                                       & \text{if}\ \ j=c, 0\leq i\leq c-1\\
                    \end{cases}\\
                    &\equiv &\begin{cases}
                            \frac{{{m-c+j}\choose i}}{{m\choose i}} & \text{if} \ \ 1\leq j\leq c-1,0\leq i\leq c-1\\
                             1                                                   & \text{if}\ \ j=c, 0\leq i\leq c-1\\
                            \end{cases}   \\
\implies \det\left(A\right) &\equiv & \frac{1}{ \underset{0\leq i\leq c-1}\Pi {m\choose i}}\det(B)
\end{eqnarray*}
 where $B = \left({{m-c+j}\choose i}\right)_{\substack{1\leq j\leq c\\ 0\leq i\leq c-1}}$.  By Lemma \ref{invmt 1}, $B$ is invertible mod $p$,  and therefore $A\in \mathrm{GL}_{c}(\mathbb{Z}_p)$.   So taking column vector $ \textbf{d} = (d_0,d_1,...,d_{c-1})^t = A^{-1}(1,,0,...,0)^t\in\mathbb{Z}^{c}_p$ gives the following
\begin{equation}{\label{eq 1 case ii m>c}}
\sum_{0\leq i\leq c-1} d_i = 1 \quad \text{for}\quad j = c 
\end{equation}
  $$\sum_{0\leq i\leq c-1}d_i \frac{{{r-(b-m+j(p-1))}\choose i}}{{m\choose i}} = 0  \quad \text{for} \quad 1\leq j\leq c-1.$$
 Now we multiply each $j^{th}$ equation ($1\leq j\leq c-1$) with $\frac{(r-m)!m!}{(b-m+j(p-1))!(r-(b-m+j(p-1)))!}$ to get
 \begin{equation}{\label{eq 2 case ii m>c}}
 \sum_{0\leq i\leq c-1}d_i\frac{{{r-i}\choose b-m+j(p-1)}}{{{r-i}\choose r-m}} = 0\quad \text{for all}\quad 1\leq j\leq c-1.
 \end{equation}
Now first, we claim that Proposition \ref{gen 2} is applicable for $0\leq i\leq c-1<m -\nu\left({{r-i}\choose r-m}\right)$.  Observe that by Lemma \ref{rk3.15} the claim is true if $(m, i)\not= (c, c-1)$ since in this case either $m\geq c+1$ or $i\leq c-2$. If $(m, i) = (c, c-1)$  then both $m$ and $i$ are atleast  $b-c+1$.  So by Lemma \ref{rk3.15} we have $\nu\left({{r-i}\choose r-m}\right) = 0$ proving the claim in this case also. Therefore taking $f = \sum_{0\leq i \leq c-1}d_if^{i}$, where $f^{i}$ are as in Proposition \ref{gen 2}(i), gives $(T-a_p)(f) \equiv\left[ g^0_{2,0}, F_m(x,y)\right]$.

\textbf{ Case (iii)} $(p-1)+b-c\leq m<2(p-1)+b-c$ and $(b, \ m)\not= (2c-p+1, \ c)$\\
Observe that in this case $c\geq 2$, and if $c = 2$ then by Remark \ref{rk 4.5} we can take $l = 0$ in  Proposition \ref{gen 2}(i) giving $(T-a_p)(f^0) \equiv\left[ g^0_{2,0}, F_m(x,y)\right]$ for above values of $m$. 
 For $c\geq 3$, we consider  the following matrix $A = (a_{j,i})\in M_{c-1}(\mathbb{Z}_p)$ given by 
\[
   a_{j,i} =     
                    \begin{cases}
                       \frac{{{r-(b-m+j(p-1))}\choose i}}{{m\choose i}} & \text{if}\ \ 2\leq j\leq c-1,\ 0\leq i\leq c-2\\                       
                     1           & \text{if} \ \  j = c,\ 0\leq i\leq c-2 \\
                  
                         \end{cases}
   \]  
  \[   
 \implies  \ \ \ \ \ \ \ \ \ \ a_{j,i} \equiv 
                            \begin{cases}
                            \frac{{{m-c+j}\choose i}}{{m\choose i}} & \text{if} \ \ 2\leq j\leq c-1,\ 0\leq i\leq c-2\\
                            1            & \text{if}\ \ j=c,\ 0\leq i\leq c-2.\\ 
                           
                            \end{cases}        \hspace{8em}
 \]\\
 Observe that 
\begin{eqnarray*}
\det \left(A\right) &\equiv& \frac{1}{\underset{0\leq i\leq c-2}\Pi{m\choose i}} \det\left(\left({{m-c+j}\choose i}\right)_{\substack{2\leq j\leq c\\ 0\leq i\leq c-2}}\right)\\
&\equiv&\frac{1}{\underset{0\leq i\leq c-2}\Pi{m\choose i}}\\
&\not\equiv& 0\bmod p
 \end{eqnarray*}
 as $\det\left(\left({{m-c+j}\choose i}\right)_{\substack{2\leq j\leq c\\ 0\leq i\leq c-2}}\right) = 1$.  Latter claim follows from Lemma \ref{grinberg} after replacing $j$ by $j-1$ and $i$ by $i+1$, and thus $A\in \mathrm{GL}_{c-1}(\mathbb{Z}_p)$.
 Proceeding exactly as in Case (ii) above (with $0\leq i\leq c-2$, $2\leq j\leq c-1$) and taking $f = \sum_{0\leq i \leq c-2}d_if^{i}$, where $f^{i}$ are as in Proposition \ref{gen 2} (ii), we have $(T-a_p)(f) \equiv\left[ g^0_{2,0}, F_m(x,y)\right]$.  We note that Proposition \ref{gen 2}(i) is applicable since $c-2<m-\nu({{r-i}\choose {r-m}})$ by Lemma \ref{rk3.15} for $0\leq i\leq c-2$.   
 
\textbf{Case (iv)} $(b, \ m) = (2c-p+1, \ c)$\\
In this case consider the following matrix $A = (a_{j,i})$ where 
\[a_{j,i} = 
\begin{cases}
\frac{{{r-(b-m+j(p-1))}\choose i}}{{m\choose i}} & \text{if}\quad 1\leq j\leq c-1, 0\leq i\leq c-1\\
1 & \text{if}\quad j = c, \ 0\leq i\leq c-1.
\end{cases}
\] 
Exactly same computation in Case(ii) above gives (\ref{eq 1 case ii m>c}) and (\ref{eq 2 case ii m>c}) in this case also.  So taking $f = \sum_{0\leq i \leq c-1}d_if^{i}$, where $f^{i}$ are as in Proposition \ref{gen 2}(ii),  we have $(T-a_p)(f) \equiv\left[ g^0_{2,0}, F_m(x,y)\right]$.  We note that Proposition \ref{gen 2}(ii) is applicable in the range $0\leq i\leq c-1$. This is clear for $i\leq c-2$ as $\nu\left({{r-i}\choose r-m}\right)\leq 1$ by Lemma \ref{rk3.15},  and for $ i = c-1$ this follows since ${{r-(c-1)}\choose r-m} = r-(c-1)\not\equiv 0\bmod p$. \\

Thus in each of the above cases we have shown that $P([g, \ F_m(x, y)]) = 0$ for $c+1-\epsilon\leq m\leq \lfloor\nu(a_p)\rfloor$ if $(b, \ c) \not= (p, \ 0)$ and for $2\leq m\leq \lfloor\nu(a_p)\rfloor$ if $(b, \ c) = (p, \ 0)$.  We also observe that $F_m(x,y)$ generates $\frac{V^{(m)}_r}{V^{(m+1)}_r}$ using Lemma \ref{lm3.2}. Hence Lemma \ref{lm m>c} gives our result by taking $G_m(x,y) = F_m(x,y)$.
\end{proof}

\begin{thm}{\label{combining}}
Let $r = s+p^t(p-1)d $, $\  s=b+c(p-1)$ such that $p\nmid d$, $2\leq b\leq p$ and $0\leq c\leq p-2$. Fix $a_p$ such that $s>2\nu(a_p)$ and $c<\nu(a_p)<\text{min}\{\frac{p}{2}+c-\epsilon,\ p-1\}$ where $\epsilon$ is defined as in (\ref{dfn epsilon}).  Further we assume $t\geq 2\nu(a_p)$ if $b\geq 2c-1$ and $t>2\nu(a_p)+\epsilon-1$ if $b\leq 2c-2$. \\
(I) If $(b,c) \not= (p,0)$ then there is a surjection 
$$\ind^G_{KZ}\left( \frac{V^{(c-\epsilon)}_r}{V^{(c+1-\epsilon)}_r}\right)\rightarrow\bar{\Theta}_{k',a_p}.$$
(II) If $(b,c) = (p,0)$ and $\nu(a_p)>1$ then there is a surjection 
$$\ind^G_{KZ}\left( \frac{V^{(1)}_r}{V^{(2)}_r}\right)\rightarrow\bar{\Theta}_{k',a_p}.$$
\end{thm}
\begin{proof}
(I) Let $\nu := \lfloor\nu(a_p)\rfloor$. Proposition \ref{m>c} gives that $\ind^G_{KZ}\left( \frac{V^{(c+1-\epsilon)}_r}{V^{(\nu+1)}_r}\right)$ is contained in $\kerp$.  Hence the desired result follows from Proposition \ref{m<slope} since $\ind^G_{KZ}\left( \frac{V^{(c+1-\epsilon)}_r}{V^{(\nu+1)}_r}\right)\subset\ind^G_{KZ}\left( \frac{V^{(c-\epsilon)}_r}{V^{(\nu+1)}_r}\right)$.\\

(II) If $(b,\ c) = (p, \ 0)$ then by Proposition \ref{m>c} we have 
$P: \ind^G_{KZ}\left( \frac{V_r}{V^{(2)}_r}\right)\rightarrow\bar{\Theta}_{k',a_p}$.  Using (\ref{r' = p}) with $n = 0$ gives that $\indkg\left(\frac{V_r}{V^{(1)}_r}\right)$ is generated by $x^r$ and $x^{r-1}y$. Observe that both of these monomials are in $\kerp$ by Remark $4.4$ of \cite{KBG}.  Therefore, 
$P$ factors through $\ind^G_{KZ}\left( \frac{V^{(1)}_r}{V^{(2)}_r}\right)$ as desired. 
\end{proof}
\pagebreak
\section{Main Results}

\begin{lemma}{\label{vrc 1}}
Let $k' = r+2,\  r=s+p^t(p-1)d$ where $s =  b+c(p-1), \ p\nmid d, \ 2\leq b\leq p, \  0\leq c\leq p-2,  \ 1\leq t$,  and $0\leq n\leq p-1$. If the map
\begin{equation}{\label{map p}}
P:\ind^G_{KZ}\left( \frac{V^{(n)}_r}{V^{(n+1)}_r}\right)\rightarrow\bar{\Theta}_{k',a_p}
\end{equation}
 is surjection.  Further if $(b,n)\not\in\{(p-2, \ 0), (p,0),(p,1)\}$ and also $b\not\in\{2n\pm 1, 2(n+1)-p, 2n-p\}$ then 
 
 $$\bar{V}_{k',a_p}\cong \begin{cases}
                                              \ind\left(\omega^{b+ n(p-1) +1}_2\right) & \text{if}\quad 2n+1\leq b\leq p\\
                                              \ind\left(\omega^{b+ (n+1)(p-1) +1}_2\right)) & \text{if}\quad 2n+1-(p-1)\leq b\leq 2n\\
                                             \ind\left(\omega^{b+ (n+2)(p-1) +1}_2\right) & \text{if}\quad 2(n+1)-2(p-1)\leq b\leq 2n-(p-1).\\
                                              \end{cases}$$

\end{lemma}

\begin{proof}
We begin by observing that if $a \equiv r-n(p+1)\bmod (p-1)$ where $1\leq a\leq p-1$  then \eqref{r' = p} and \eqref{r' not p} give
$$0\longrightarrow V_{a}\otimes D^n\longrightarrow\frac{V^{(n)}_r}{V^{(n+1)}_r}\longrightarrow V_{p-1-a}\otimes D^{a+n}\longrightarrow 0.$$
Now using Propositions 3.1-3.3 of \cite{KBG} we deduce that $P$ factors through exactly one of the subquotients above,  and that $\bar{\Theta}_{k',a_p}$ is reducible only if $a$ or $p-1-a$ equals $p-2$.  Thus,  the reducible cases occur only if $(b,n)\in\{(p-2, \ 0), (p,0),(p,1)\}$ or if $b \in\{2n\pm 1, 2(n+1)-p, 2n-p\}$.  In the generic cases where $(b,n)\not\in\{(p-2, \ 0), (p,0),(p,1)\}$ and $b \not\in\{2n\pm 1, 2(n+1)-p, 2n-p\}$ we further note that we obtain the same irreducible representation irrespective of whether the map $P$ factors through the submodule or the quotient (using the classification of smooth admissible mod $p$ representations of $\gl$). Thus we have (by Proposition 3.3 of \cite{KBG}) $\bar{V}_{k',a_p}$ as given above. 
\end{proof}

Now suppose $2\leq b\leq p$ and $0\leq c\leq p-2$. Let us define a set $E'$ of ordered pair $(b, \ c)$ as:
$$ E' = \{(p-2, \ 0), (p,0),(p,1),(2c+1,\ c), (2c-1,\ c), (2c-3, \ c),(2c-p, \ c), (2c-2-p, \ c),(2c-4-p, \ c)\}.$$
The set $E'$ denotes the set of exceptional points $(b,\ c)$ at which $\bar{\Theta}_{k', a_p}$ may be reducible.  We obtain $E'$ from Theorem \ref{combining} and Lemma \ref{vrc 1}. 

\begin{prop}{\label{final prop}}
Let $k'= r+2$  and  $k = s+2$. Assume all the hypotheses of Theorem \ref{combining}.  If $b \not\in\{2c+1, \ 2c-1, \ 2c-p, \ 2(c-1)-p\}$ and also $(b,c)\not =(p,0)$ then $\bar{V}_{k', a_p}\cong  \ind\left(\omega^{k-1}_2\right)$.
\end{prop}
\begin{proof}
Since $(b, \ c)\not= (p, \ 0)$ Theorem \ref{combining} gives
$$P:\ind^G_{KZ}\left( \frac{V^{(c-\epsilon)}_r}{V^{(c+1-\epsilon)}_r}\right)\twoheadrightarrow\bar{\Theta}_{k',a_p}.$$
For $2c-1\leq b\leq p$ and $(b, \ c)\not\in E'$ we use Lemma \ref{vrc 1} with $n = c$ to get\vspace{2mm}\\
$\hspace*{10em}\bar{V}_{k', a_p}\cong
\begin{cases}
\ind\left(\omega^{b+c(p-1)+1}_2\right) & \text{if}\quad 2c+1\leq b\leq p \vspace{2mm}\\
\ind\left(\omega^{b+c(p-1)+p}_2\right) & \text{if}\quad 2c-1\leq b\leq 2c.
\end{cases}$\\
Therefore we have $\bar{V}_{k',a_p}\cong \ind\left(\omega^{k-1}_2\right)$. In the second case this follows since we have $b=2c$ and so $\omega^{b+c(p-1)+p}_2$ is conjugate to $\omega^{k-1}_2$ (using $b=2c$, $p(k-1) - (b+c(p-1)+p) = c(p^2-1)$). The remaining cases of $b$ are also treated similarly using Lemma \ref{vrc 1} to obtain the desired reduction outside the set $E'$. Now we will deal with some of the points in $E'$.

\textbf{Cases (i)} $(b, \ c) = (p-2, \ 0)$\\
We apply (\ref{r' not p}) (with $n = 0$ and $r' = 2p-3$) to see that the image of  $\ind^{G}_{KZ}\left(V_{p-2}\right)$ in $\ind^{G}_{KZ}\left(\frac{V_r}{V^{(1)}_r}\right)$ is generated by $[1,\ x^r]$ which belongs to $\kerp$ by Remark $4.4$ of \cite{KBG}. Hence $P$ factors through $\ind^G_{KZ}\left(V_1\otimes D^{p-2}\right)$. Therefore  Proposition $3.3$ of \cite{KBG} gives $\bar{V}_{k',  a_p}\cong \ind\left(\omega^{2+(p-2)(p+1)}_2\right)$.  We conclude by observing that $\omega^{2+(p-2)(p+1)}_2$ is conjugate to $\omega^{k-1}_2$ as $k  = p$ and $2+(p-2)(p+1)= p(k-1)$. 

\textbf{ Case (ii)} $(b, \ c) = (p, \ 1)$\\
Let $f_1,f_2,f_3\in \indkg\left(\symqp\right)$ given by 
\begin{eqnarray*}
f_1 &=& \left[ 1,\frac{1}{a_p}(x^py^{r-p}-x^{r-(p-1)}y^{p-1})\right]\\
f_2 &=&\sum_{\lambda\in I^*_1} \left[ g^0_{1,\lambda},\frac{1}{\lambda^p(p-1)}(y^r-x^{r-s}y^s)\right]\\
f_3 &=& \left[ 1, \underset{\substack{s-1\leq j<r-1\\j\equiv 0\bmod (p-1)}}\sum{r\choose j}x^{r-j}y^j\right].
\end{eqnarray*}
Now
\begin{align*}
	T^+(f_1) &=& \sum_{\mu\in I^*_1}\left[ g^0_{1,\mu}, \sum_{0\leq j\leq p-1}\frac{p^j(-\mu)^{r-p-j}}{a_p}\left( {{r-p}\choose j}-{{p-1}\choose j}\right)x^{r-j}y^j\right]\\
	&&+ \sum_{\mu\in I_1}\left[ g^0_{1,\mu}, \sum_{p\leq j\leq r-p}\frac{p^j{{r-p}\choose j}(-\mu)^{r-p-j}}{a_p}x^{r-j}y^j\right]\\
	&&-\left[g^0_{1,0},\frac{p^{p-1}}{a_p}x^{r-(p-1)}y^{p-1}\right].
\end{align*}
Here we observe that first sum is zero mod $p$ because for $j\geq 1$,  $j+t-\nu(j!)-\nu(a_p)\geq t+1-\nu(a_p)>0$ as $\nu\left({{r-p}\choose j}-{{p-1}\choose j}\right)\geq t-\nu(j!)$ and the last two summation are zero mod $p$ as $j-\nu(a_p)>0$ for $j\geq p-1$. 
$$T^-(f_1) = \left[\alpha,\frac{p^p}{a_p}x^py^{r-p}-\frac{p^{r-(p-1)}}{a_p} x^{r-(p-1)}y^{p-1}\right]$$
Here we note that $p-\nu(a_p)>0$ and $r-(p-1)\geq p$.  Therefore we have $T^+(f_1),T^-(f_1)$ both are zero mod\ $p$. Hence 
\begin{equation}{\label{f'_1}}
	(T-a_p)(-f_1) = \left[ 1,(x^py^{r-p}-x^{r-(p-1)}y^{p-1})\right].
\end{equation}
Now
\begin{align*}
 &T^+\left( \left[ g^0_{1,\lambda},\frac{1}{\lambda^p(p-1)}(y^r-x^{r-s}y^s\right]\right)=\sum_{\mu\in I^*_1}\left[ g^0_{2,\lambda+p\mu}, \sum_{0\leq j\leq s}\frac{p^j(-\mu)^{r-j}}{\lambda^p(p-1)}\left( {{r}\choose j}-{{s}\choose j}\right)x^{r-j}y^j\right]
\end{align*}
$$ \hspace{15em} + \sum_{\mu\in I_1}\left[ g^0_{2,\lambda +p\mu}, \sum_{s+1\leq j\leq r}\frac{p^j{{r}\choose j}(-\mu)^{r-j}}{\lambda^p(p-1)}x^{r-j}y^j\right]$$
$$\hspace{22em}-\left[g^0_{2,\lambda},\frac{p^s}{\lambda^p(p-1)}x^{r-s}y^s\right].$$
Here we observe that $T^+(f_2)\equiv 0\bmod p$.
\begin{eqnarray*}
	T^-\left( \left[ g^0_{1,\lambda},\frac{1}{\lambda^p(p-1)}(y^r-x^{r-s}y^s\right]\right) &=& \left[1, \sum_{0\leq j\leq r}\frac{{r\choose j}\lambda^{r-j}}{(p-1)\lambda^p}x^{r-j}y^j\right]\\
	&&- \left[1, \sum_{0\leq j\leq s}\frac{p^{r-s}{s\choose j}\lambda^{s-j}}{(p-1)\lambda^p}x^{r-j}y^j\right].\\
	\end{eqnarray*}
Hence we have
$$ T^-\left( \left[ g^0_{1,\lambda},\frac{1}{\lambda^p(p-1)}(y^r-x^{r-s}y^s\right]\right) = \left[1, \sum_{0\leq j\leq r}\frac{{r\choose j}\lambda^{r-j}}{(p-1)\lambda^p}x^{r-j}y^j\right]\hspace{1em}(\text{as}\quad r-s>0)\\$$
	\begin{eqnarray*}
	\implies T^-(f_2) &=& \left[1, \underset{\substack{0\leq j\leq r\\ j\equiv\ 0\bmod (p-1)}}\sum{r\choose j}x^{r-j}y^j\right]\\
	(T-a_p)(f_2) &=& \left[1, \underset{\substack{0\leq j\leq r\\j\equiv\ 0\mod (p-1)}}\sum{r\choose j}x^{r-j}y^j\right]\\
	\implies (T-a_p)(f_2) &=&\left[1,\ x^r\right]+\left[1,\ {r\choose p-1}x^{r-(p-1)}y^{p-1}\right]+ f_3 \\
	&&\hspace*{8em}+\left[1,\ {r\choose r-1}xy^{r-1}\right].
\end{eqnarray*}
Now note $r = p+p-1+p^t(p-1)d\implies {r\choose p-1}\equiv 1\bmod p$ by Lucas formula and ${r\choose r-1} = r\equiv -1\bmod p$.
\begin{equation}{\label{f'_2}}
	\implies (T-a_p)(f_2) =\left[1,\ x^r\right]+\left[1,\ x^{r-(p-1)}y^{p-1}\right]+ f_3 -\left[1,\ xy^{r-1}\right]
\end{equation}
\pagebreak\\
$$T^+\left(\frac{ f_3}{a_p}\right) = \sum_{\mu\in I^*_1}\left[ g^0_{1,\mu},\sum_{0\leq j\leq r}\frac{p^j(-\mu)^{r-1-j}}{a_p}\underset{\substack{s-1\leq i<r-1\\i\equiv 0\bmod (p-1)}}\sum{r\choose i}{i\choose j}x^{r-j}y^j\right]$$
$$\hspace{17em}+ \left[ g^0_{1,0},\underset{\substack{s-1\leq j<r-1\\ j\equiv 0\bmod (p-1)}}\sum\frac{p^j{r\choose j}}{a_p}x^{r-j}y^j\right].$$
Here we note that $j-\nu(a_p)>0$ for $j\geq p-1$ this gives that the first summation truncates to $j\leq p-2$ and the second summation is zero mod\ $p$.
$$\implies \ \ \ \ \ \ \  T^+\left(\frac{ f_3}{a_p}\right) = \sum_{\mu\in I^*_1}\left[ g^0_{1,\mu},\sum_{0\leq j\leq p-2}\frac{p^j(-\mu)^{r-1-j}}{a_p}S_{r,j,0,1}x^{r-j}y^j\right]$$
Since $c+m = 2$, so Lemma \ref{srjm} gives $\nu(S_{r,j,0,1})\geq t-c+1$ therefore  $T^+\left(\frac{ f_3}{a_p}\right)\equiv\ 0\bmod p$ as $t\geq 2\nu(a_p)$.
$$T^-\left(\frac{ f_3}{a_p}\right) = \left[ \alpha,\underset{\substack{s-1\leq j<r-1\\ j\equiv 0\bmod (p-1)}}\sum\frac{p^{r-j}}{a_p}x^{r-j}y^j\right]$$
Note that $r-j-\nu(a_p)\geq p-\nu(a_p)>0 \ \ \ \ \ \ \implies T^-\left(\frac{ f_3}{a_p}\right)\equiv \ \ 0\bmod p$ 
\begin{equation}{\label{f'_3}}
	(T-a_p)\left(\frac{ f_3}{a_p}\right) = - f_3
\end{equation}
Since $\nu(a_p)>1$,  using Remark of \cite{KBG} there exist $f_0,\in \ind^G_{KZ}\left(Sym^r(\bar{\mathbb{Q}}^2_p)\right)$ such that 
\begin{equation}{\label{f'_4}}
	(T-a_p)(f_0) = \left[1, x^r\right]
\end{equation}
Now take $f = -f_1+f_2+\left(\frac{ f_3}{a_p}\right) -f_0$ then \eqref{f'_1},\ \eqref{f'_2},\ \eqref{f'_3}, \ \eqref{f'_4} imply 
$$\hspace{2em}(T-a_p)(f) = \left[ 1,(x^py^{r-p}-xy^{r-1})\right]$$
\begin{equation*}
	\implies \ \ \ \ \ \ \ \ \ \ \ \ (T-a_p)(f) = \left[1,\theta y^{r-(p+1)}\right]. \ \ \ \ \ \ \ \ \ \ \ \ \ \ \ \ \ \ \ \ \ \ \
\end{equation*}
Hence $[1, \ \theta y^{r-(p+1)}]\in\kerp$. Now we observe that (\ref{r' not p}) (with $n = 1$ and $r' = 2p-3$) gives  that the image of $\ind^{G}_{KZ}\left(V_{p-2}\otimes D\right)$ in $\ind^{G}_{KZ}\left(\frac{V^{(1)}_r}{V^{(2)}_r}\right)$ is generated by $[1, \ \theta y^{r-(p+1)}]$ which is contained in $\kerp$. Therefore the map $P$ factors through $\ind^G_{KZ}\left(V_1\right)$.  Hence by using Proposition $3.3$ of \cite{KBG} we have $\bar{V}_{k',a_p}\cong\ind\left(\omega^2_2\right)$.  Our claim follows since $\omega^2_2$ is conjugate to $\omega^{2p}_2$ (here $k-1 =  2p$).

 \textbf{Case (iii)} $b = 2c-3$\\
 In this case we note that (\ref{r' not p}) (with $n = c-1$ and $r' = 2p-3$) gives that the image of $\ind^{G}_{KZ}\left(V_{p-2}\otimes D^{c-1}\right)$ in $\ind^{G}_{KZ}\left(\frac{V^{(c-1)}_r}{V^{(c)}_r}\right)$ is generated by $[1,\ \theta^{(c-1)}x^{r-(c-1)(p+1)}]$.  The latter belongs to $\kerp$ since 
 $$\theta^{(c-1)} x^{r-(c-1)(p+1)} = \underset {0\leq i\leq c-1}\sum (-1)^i{{c-1}\choose i}x^{r-(b-(c-2)+i(p-1))}y^{b-(c-2)+i(p-1)}$$
 and so every monomial on the right is in $\kerp$ by taking $m = c-2$ in Proposition \ref{mono 1}.  
Hence $P$ factors through $\ind^G_{KZ}\left(V_1\otimes D^{(c-2)}\right)$.  Therefore  Proposition $3.3$ of \cite{KBG} gives $\bar{V}_{k',a_p}\cong\ind\left(\omega^{2+(c-2)(p+1)}_2\right)$. Hence we have our result because $\omega^{2+(c-2)(p+1)}_2$ is conjugate to $\omega^{k-1}_2$ (as $k-1 = c(p+1)-2$ and  $p(k-1) - 2-(c-2)(p+1) = c(p^2-1)$).

 \textbf{Case (iv)} $b = 2(c-2)-p$\\
 In this case we note that by using (\ref{r' not p}) (with $n = c-2$ and $r' = 2p-3$) the image of $\ind^{G}_{KZ}\left(V_{p-2}\otimes D^{c-2}\right)$ in $\ind^{G}_{KZ}\left(\frac{V^{(c-2)}_r}{V^{(c-1)}_r}\right)$ is seen to be generated by $[1,\ \theta^{(c-2)}x^{r-(c-2)(p+1)}]$ which belongs to $\kerp$. This is clear by taking $m = c-3$ in Proposition \ref{mono 1.2} and observing that 
 $$\theta^{(c-2)} x^{r-(c-1)(p+1)} = \underset {0\leq i\leq c-2}\sum (-1)^i{{c-2}\choose i}x^{r-(b-(c-3)+(i+1)(p-1))}y^{b-(c-3)+(i+1)(p-1)}.$$
Hence $P$ surjects from $\ind^G_{KZ}\left(V_1\otimes D^{(c-3)}\right)$. Therefore  Proposition $3.3$ of \cite{KBG} gives  $\bar{V}_{k',a_p}\cong\ind\left(\omega^{2+(c-3)(p+1)}_2\right)\cong\ind\left(\omega^{k-1}_2\right)$. 
\end{proof}

\begin{cor}{\label{vkap}}
Let $p\geq 7$ be a prime and $k = s+2$. Assume all the hypotheses of Theorem \ref{combining}. If we further assume $b \not\in\{2c+1, \ 2c-1, \ 2c-p, \ 2(c-1)-p\}$ and $(b,c)\not =(p,0)$ then  $\bar{V}_{k,a_p}\cong \ind\left(\omega^{k-1}_2\right)$.
\end{cor}
\begin{proof}
We begin by observing that if $\nu(a_p)>c+1$ then the conclusion follows by \cite{BLZ} (note that $ p+1 \nmid k-1$ from hypothesis).  So from now on we will assume $\nu(a_p)\leq c+1$.  Observe that since $\nu(a_p)\leq c+1$ we have
\begin{eqnarray*}
3\nu(a_p)+\frac{(k-1)p}{(p-1)^2}+1 &\leq & 4(c+1)+\frac{b+1}{(p-1)}+\frac{k-1}{(p-1)^2}\\
&< &\begin{cases}
4(c+1)+2 & \text{if}\quad 2\leq b\leq p-3\\
4(c+1)+3 & \text{if}\quad p-2\leq b\leq p.
\end{cases}
\end{eqnarray*}
The last inequality follows as $k\leq (p-1)^2+3$ and $p\geq 5$.  If $c=0$ then $\bar{V}_{k,a_p}\cong \ind\left(\omega^{k-1}_2\right)$ by \cite{Br03b} as $k \leq p+1$. Therefore,  assuming $c \geq 1$ and $p\geq 7$ we get $k-4(c+1)\geq b$, giving us $k>3\nu(a_p)+\frac{(k-1)p}{(p-1)^2}+1$.  So by Theorem \ref{Berger} there exists a constant $m = m(k, a_p)$ such that for all $k'' \in k+p^{m-1}(p-1)\mathbb{Z}^{\geq 0}$ we have $\bar{V}_{k'',a_p}\cong \bar{V}_{k,a_p}$.  For $t$ as in Proposition \ref{final prop} we have $\bar{V}_{k',a_p}\cong \ind\left(\omega^{k-1}_2\right)$ for $k'\in k+p^t(p-1)\mathbb{N}$. Hence these two facts together gives $m(k, a_p)\leq t+1$, and so we have the desired result.
\end{proof}
\noindent We have our main result below.
\begin{thm}{\label{main result}}
Let $k =b+c(p-1)+2$ with $2\leq b\leq p$ and $0\leq c\leq p-2$. Fix $a_p$ such that  $s>2\nu(a_p)$ and $c<\nu(a_p)<\text{min}\{\frac{p}{2}+c-\epsilon,\ p-1\}$ where $\epsilon$ is defined as in (\ref{dfn epsilon}). Further if $b \not\in\{2c+1, \ 2c-1, \ 2c-p, \ 2(c-1)-p\}$ and $(b,c)\not =(p,0)$ then the Berger's constant $m(k,a_p)$ exists such that $m(k,a_p)\leq\lceil 2\nu(a_p)\rceil+\epsilon+1$.  Moreover,  $\bar{V}_{k',a_p}\cong \ind\left(\omega^{k-1}_2\right)$ for all $k' \in k+p^{t}(p-1)\mathbb{Z}^{\geq 0}$,  where $ t \geq \lceil 2\nu(a_p)\rceil+\epsilon$.
\end{thm}

\noindent\textbf{Acknowledgements.}
We owe a great debt to the work in \cite{maam}, and also acknowledge the results in \cite{Berger12} and \cite{BLZ} critical to our work. The authors would like to express sincere gratitude to Shalini Bhattacharya for giving useful suggestions regarding this problem. The second author acknowledges the support received from the NBHM (under DAE, Govt. of India) Ph.D.  fellowship grant $0203/11/2017$/RD-II/$10386$. \\


\begin{thebibliography}{ABCD}
\bibitem [BL94]{BL94}
L. Barthel and R. Livn\'e.
\newblock Irreducible modular representations of GL2 of a local field.
\newblock Duke Math. J. 75, no. 2:261-292, 1994.
\bibitem [BL95]{BL95}
L. Barthel and R. Livn\'e.
\newblock Modular representations of GL2 of a local field: the ordinary, unramified case
\newblock J. Number Theory 55 no. 1:1-27,  1995.
\bibitem [B10]{B10}
L.  Berger.
\newblock Repr\'esentations modulaires de $\gl$ et repr\'entations galoisiennes de dimension 2.
\newblock Ast\'erisque, 330:263–279, 2010.
\bibitem[B12]{Berger12}
 L.  Berger.
 \newblock Local constancy for the reduction mod p of 2-dimensional crystalline representations.
 \newblock {\em Bull. London Math. Soc.}, 44(3): 451-459, 2012.
 \bibitem[B]{Berger}
 L.  Berger.
 \newblock Errata for my articles, \url{http://perso.ens-lyon.fr/laurent.berger/articles.php}.
 \bibitem [BB10]{BB10}
L.  Berger and C.  Breuil. 
\newblock Sur quelques repr\'esentations potentiellement cristalline de $\gl$.
\newblock Ast\'erisque, 330:155–211, 2010.
 \bibitem[BLZ04]{BLZ}
L.  Berger, H.  Li and H.  Zhu.
\newblock Construction of some families of 2-dimensional crystalline representations.
\newblock {\em  Math. Ann}, 329:365--377, 2004.
\bibitem [SB20]{maam}
S. Bhattacharya.
\newblock  Reduction of certain crystalline representations and local constancy in the weight space.
\newblock Journal de Théorie des Nombres de Bordeaux, Tome 32(1):25-47,  2020.
\bibitem [BG15]{BG}
S. Bhattacharya and E. Ghate.
\newblock Reductions of Galois representations for slopes in $(1, 2)$. 
\newblock Doc. Math.20: 943-987,  2015.
\bibitem [BGR18]{BGR18}
S. Bhattacharya,  E. Ghate and S.  Rozensztajn.
\newblock Reductions of Galois representations for slopes in $1$. 
\newblock J. Algebra , 508:98-156,  2018.
\bibitem[B03a]{Br03a}
C.  Breuil.
\newblock Sur quelques repr\'esentations modulaires et $p$-adiques de ${\mathrm {GL}}_2({\mathbb Q}_p)$. I.
\newblock Compos. Math. 138, no. 2:165-188, 2003.
\bibitem[B03b]{Br03b}
C.  Breuil.
\newblock Sur quelques repr\'esentations modulaires et $p$-adiques de ${\mathrm {GL}}_2({\mathbb Q}_p)$. II.
\newblock J. Inst. Math. Jussieu, 2:23–58,  2003.
\bibitem[BG09]{KBG}
K.  Buzzard and T.  Gee.
\newblock Explicit reduction modulo $p$ of certain two-dimensional crystalline representations.
\newblock {\em Int. Math. Res. Notices},  no. 12,  2303--2317,  2009.
\bibitem [CGY21]{cgy}
A. Chitrao,  E. Ghate  and S. Yasuda
\newblock {Semi-stable representations as limits of crystalline representations. arXiv preprint,  2021.}
\bibitem[CF00]{cf}
P. Colmez \& J.-M. Fontain.
\newblock Construction des représentations $p$-adiques semi-stables.
\newblock Invent. Math. 140:1-43,  2000.
\bibitem [DG]{grinberg}
Darij Grinberg.
\newblock A hyperfactorial divisibility,\\ \url{https://www.cip.ifi.lmu.de/~grinberg/hyperfactorialBRIEF.pdf}.
\bibitem[E92]{be92}
B. Edixhoven.
\newblock The weight in Serre's conjectures on modular forms.
\newblock Invent. Math. 109:563-594,  1992.
\bibitem [GG15]{GG15}
A. Ganguli and E. Ghate.
\newblock Reductions of Galois representations via the mod p Local Lang-
lands Correspondence.
\newblock J. Number Theory 147:250-286,  2015.
\bibitem [GR20]{gv}
E. Ghate and V. Rai
\newblock{Reductions of Galois representations of Slope $\frac{3}{2}$.  arXiv preprint, 2020.}
\bibitem [GV22]{GhR}
E. Ghate and R. Vangala.
\newblock The Monomial Lattice in Modular Symmetric Power Representations. 
\newblock {Algebr.  Represent. Theory 25},  no. 1,  121-185,  2022.
\bibitem [G78]{glover}
 D. J. Glover.
\newblock A study of certain modular representations.
\newblock J. Algebra 51:425-475, 1978.
\bibitem[K68]{k68}
G. S. Kazandzidis.
\newblock Congruences on binomial coefficients.
\newblock {\em Bull. Soc. Math. Gr\'ece{(NS)}},9:1-12, 1968.

\bibitem [R18]{sandra}
S. Rozensztajn
\newblock {An algorithm for computing the reduction of 2-dimensional crystalline representations of Gal$\left(\bar{\mathbb{Q}}_p/\mathbb{Q}_p\right)$}
\newblock {International Journal of Number Theory},1857-1894,2018



\end{thebibliography}
\end{document}